\documentclass[12pt]{amsart}
\usepackage{mathrsfs}
\usepackage{stmaryrd}
\usepackage[alphabetic]{amsrefs}
\usepackage[all]{xy}

\title[On heights of motives with semistable reduction]
{On heights of motives with semistable reduction}

\author{Teruhisa Koshikawa}
\address{Department of Mathematics, University of Chicago}
\email{teruhisa@uchicago.edu}

\newcommand\cD{\mathcal D}

\newcommand\cL{\mathcal L}
\newcommand\cM{\mathcal M}
\newcommand\cO{\mathcal O}

\newcommand\cV{\mathcal V}

\newcommand\fM{\mathfrak M}
\newcommand\fN{\mathfrak N}
\newcommand\fS{\mathfrak S}

\newcommand\bA{\mathbb A}
\newcommand\bC{\mathbb C}
\newcommand\bF{\mathbb F}

\newcommand\bN{\mathbb N}
\newcommand\bQ{\mathbb Q}
\newcommand\bR{\mathbb R}

\newcommand\bZ{\mathbb Z}

\DeclareMathOperator{\ab}{ab}

\DeclareMathOperator{\BK}{BK}

\DeclareMathOperator{\crys}{crys}
\DeclareMathOperator{\Deg}{Deg}

\DeclareMathOperator{\dR}{dR}
\DeclareMathOperator{\et}{\acute{e}t}
\DeclareMathOperator{\Frac}{Frac}
\DeclareMathOperator{\Fil}{Fil}
\DeclareMathOperator{\free}{fr}
\DeclareMathOperator{\Frob}{Frob}
\DeclareMathOperator{\Gal}{Gal}
\DeclareMathOperator{\gr}{gr}

\DeclareMathOperator{\Hom}{Hom}

\DeclareMathOperator{\Ind}{Ind}

\DeclareMathOperator{\rk}{rk}
\DeclareMathOperator{\SL}{SL}

\DeclareMathOperator{\st}{st}

\DeclareMathOperator{\torsion}{tor}
\DeclareMathOperator{\Ver}{Ver}

\theoremstyle{plain}
\newtheorem{thm}{Theorem}[section]
\newtheorem{lem}[thm]{Lemma}
\newtheorem{prop}[thm]{Proposition}
\newtheorem{cor}[thm]{Corollary}

\theoremstyle{definition}
\newtheorem{rem}[thm]{Remark}
\newtheorem{defn}[thm]{Definition}
\newtheorem{exam}[thm]{Example}

\newtheorem{conj}[thm]{Conjecture}

\begin{document}

\begin{abstract}
We study heights of motives with integral coefficients over number fields introduced by Kato. 
It is a generalization of the Faltings height of an abelian variety and we establish generalizations of some properties of the Faltings height in our context as conjectured by Kato. 
This sheds some light on integral structures of motives. 
\end{abstract}

\maketitle

\setcounter{tocdepth}{1}
\tableofcontents

\section{Introduction}

In his celebrated paper ~\cite{Faltings83}, Faltings established a number of significant results at the same time. 
In particular, he proved the Tate conjecture and the Shafarevich conjecture for abelian variteties, and the semisimplicity of Tate modules. 

An important notion introduced in \textit{ibid.} is so-called the Faltings height of an abelian variety. 
He assigned a real number $h(A_F)$ to an abelian varietiy $A_F$ over a number field $F$. 
There are two key properties of the Faltings height:
\begin{itemize}
\item Fix a real number $c$. There are only finitely many isomorphism classes of abelian varieties over $F$ with heights less than $c$. (Finiteness)
\item Fix an abelian variety $A_F$ over $F$ with semistable reduction everywhere. The set of heights of abelian varieties over $F$ which are isogeneous to $A_F$ is finite. (Boundedness)
\end{itemize}

Faltings essentially proved these two properties in \textit{ibid.}, and deduced such significant results from them. 
(Technically speaking, his actual arguments are different.)  
The first property above can be proved by studying the moduli space of abelian varieties, and it is a consequence of the fact that heights of abelian varieties in the universal family behaves as the height function attached to a line bundle with a metric of logarithmic singularities.

Our aim in this paper is to generalize these results to any pure motives with integral coefficients as much as possible. 
It is Kazuya Kato ~\cite{Kato1} who first gave a candidate for the height of a pure $\bZ$-motive. 
We essentially follow his definition but we modify it slightly. 
When we define a height, it is important to choose a good integral structure of the de Rham realization, and our choice is (possibly) different from his one. 

A little more precisely, given a pure $\bZ$-motive with semistable reduction everywhere, integral $p$-adic Hodge theory, mainly developed by Breil, Kisin and Liu, provides such an integral structure.  
We use it to define our height. 
(Since we need to work with semistable representations, we can only define the stable height in general.)

As our definition uses integral $p$-adic Hodge theory, there are many techniques we can use to study heights, and we can indeed prove some generalizations of the results of Faltings we mentioned before. 
(We cannot prove any of big conjectures though.)

We can prove the boundedness completely. 

\begin{thm}
Fix a pure $\bZ$-motive $M$ over a number field $F$ with semistable reduction everywhere. 
The set of heights of pure $\bZ$-motives which are isogeneous to $M$ is finite. 
\end{thm}

The key point of the proof is finding enough isogenies which preserve heights. 
This is done by generalizing Raynaud's ~\cite{Raynaud} refinement of Faltings' arguement. 
Raynaud used the theory of finite flat group schemes.  
In our case, the theory of torsion Kisin modules is a replacement. 
Note that we can actually give an effective bound on the size of the set above. 

Another result we obtain is, for instance, 

\begin{thm}
Let $C$ be a quasi-projective smooth curve over a number field $F$, and fix a real number $c$ and an integer $i$. 

Given a projective smooth morphism $f\colon X\to C$ and a rational point $t$ of $C$, we write $M_t$ for the $\bZ$-motive attached to the $i$-th cohomology of $X_t$ modulo torsion. 
We assume $f$ satisfies some technical assumptions on degenerations, cf.~$(\ast)$ in Section \ref{Heights in families}. 
In particular, $f$ has log smooth degenerations. 

If the variation of Hodge structure attached to the family $M_t$ is nonisotrivial on any geometric connected component, then there are only finitely many rational points at which the (stable) heights of $M_t $ are less than $c$. 
\end{thm}

Our definition of the height uses an \emph{internal integral} $p$-adic Hodge theory, in the sense that it does not involve geometry. 
(We certainly use geometric results on \emph{rational} $p$-adic Hodge theory.)
So, to prove the theorem above, we need to relate it to geometry. 
We call such a theory an \emph{external integral} $p$-adic Hodge theory. 
It is less complete compared to the internal one, but enough developed for our purpose. 

Another ingredient for the theorem above is a result on variations of Hodge structure ~\cite{Peters}. 
It fits nicely with Kato's definition of the height. 

We can prove a variant with $C$ replaced by a higher dimensional projective smooth variety. 
There is a corresponding result ~\cite{Griffiths} on variations of Hodge structure in this case.  

Finally, we outline the structure of the paper, which has two parts. 
In the first part, we begin with reviews on integral $p$-adic Hodge theory (Section $2$ and $3$). We give the definition of the height of a pure motive in Section $4$ and then discuss some basic properties. 
In Section $5$, we review the results on variations of Hodge structure, and apply it to study heights of motives in a family (Section $6$). 

To show the boundedness, which is a goal of the second part, we discuss torsion $p$-adic Hodge theory in Section $7$. 
In Section $8$ we analyze a class of isogenies as a preparation. 
We prove the boundedness in Section $9$. 

An additional Section $10$ provides a conjectural picture. 

\subsection*{Acknowledgments}
I am deeply indebted to my advisor, Kazuya Kato, for sharing his ideas and his constant encouragement. 

I would like to thank Naoki Imai, Kazuya Kato and Takuya Maruyama for helpful comments. 
I am grateful to Keerthi Madapusi Pera for discussion. 

\subsection*{Notation and Conventions}
Throughout this paper, $K$ is a nonarchimedean local field of mixed characteristic and $F$ is a number field. 
For a finite place $v$ of $F$, we will set $K=F_v$ to apply local results.  
We will write $O_F$ and $O_K $ for the rings of integers of $F$ and $K$ respectively. 

Choosing algebraic closures, we let $G_K$ and $G_F$ be the absolute Galois groups of $K$ and $F$ respectively.  
We denote by $I_K$ the inertia subgroup of $G_K$. 
For a place $v$ of $F$, we will write $F_v$ for the completion of $F$ at $v$, and 
$G_v\supset I_v$ for its absolute Galois group and the inertia subgroup. 
We will fix an embedding  $G_v\to G_F$. 

We will choose a uniformizer $\pi$ of $K$ and its $p^n$-power roots $\pi_n$ of $\pi$ as in Subsection $2.1$. 
We define $K_{\infty}=\bigcup_n K(\pi_n)$ inside the algebraic closure of $K$.  
We will denote by $G_\infty$ its absolute Galois group. 

Our conventions on $p$-adic Hodge theory is explained in Section $2$. 
In particular, our normalization of the Hodge-Tate weights is explained in Definition~\ref{Hodge-Tate}. 

We use the language of log schemes in the sense of Fontaine-Illusie and Kato. 
In this paper, a log scheme means fine and saturated. In particular, fiber products will be taken in the category of fine saturated log schemes. 
The canonical log structure on $O_K$ is given by the special fiber. 

We let $\hat{\bZ}$ be the profinite completion of $\bZ$. 
We will write $\bA^f_{\bQ}$ for the ring $\bQ\otimes_\bZ \hat{\bZ}$ of finite adeles of $\bQ$. 

Let $p$ be a prime number. 
If a ring $A$ is given the $p$-adic topology, we will write $\widehat{A}$ for the $p$-adic completion of $A$, and $A_m$ for $A/p^mA$. 
We say a free $A_m$ module $M$ is free of level $m$. 
Its rank will be denoted by $\rk_{A_m} M$. 

If $A$ is equipped with an endomorphism $\varphi$, a Frobenius endomorphism of an $A$-module $M$ means a semilinear endomorphism with respect to $\varphi$. 
We will write $A\otimes_{\varphi,A}M$ or $\varphi^*M$ for the module obtained by scalar extension of $M$ under $\varphi$.  
If $A$-modules $M$ and $M'$ have Frobenius endomorphisms, $\Hom_{\varphi}(M,M')$ is the set of homomorphisms compatible with Frobenius endomorphisms. 
More generally, $\Hom$ with a subscript means that it is the set of morphisms which respect the corresponding structure. 

Suppose $A$ is a $p$-adic completed domain and we let $D$ be a finite free $\bQ_p\otimes_{\bZ_p}A$-module. 
An $A$-lattice in $D$ is a finite free $A$-submodule $M$ of $D$ such that $\bQ_p\otimes_{\bZ_p}M$ is canonically isomorphic to $D$.  
If $A=\bZ_p$, we simply call it a lattice. 

If $A$ is a ring and $G$ is a group, an $A$-representation of $G$ is $A$-module with an action of $G$. 

An object obtained by scalar extension or base change will be often denoted by adding a suitable subscript to the original object (or replace the subscript by new one). 

We will generally write $M(d)$ for the $d$-th Tate twist of an object $M$ in any context. 

\section{Internal Integral $p$-adic Hodge Theory}
We review internal aspects of integral $p$-adic Hodge theory, which study lattices inside a semisstable Galois representation. 

\subsection{Rational theory}
Before discussing  the integral theory, we need to recall the rational theory briefly.  

We let $k$ be a finite field of characteristic $p>0$ and we write $K_0=\Frac W$ for the fraction field of the ring $W=W(k)$ of Witt vectors. 
It has a natural Frobenius lift $\varphi$. 
We fix an algebraic closure $\overline{K}$ of $K_0$ with the subring $O_{\overline{K}}$ of integers, and take a finite extension $K$ of $K_0$ inside $\overline{K}$. 
We denote by $G_K=\Gal(\overline{K}/K)$ the absolute Galois group of $K$. 

We first recall some period rings of Fontaine. 
Let $R=\varprojlim O_{\overline{K}}/p$; limit is taken with respect to $p$-powers.  
There is a natural homomorphism from the ring of Witt vectors $W(R)$ to $\widehat{O}_{\overline{K}}$, denoted by $\theta$. 
The ring $A_{\crys}$ is the $p$-adic completion of the PD hull of $\theta$. 
It is equipped with a filtration and a Frobenius endomorphism $\varphi$. 
We write $t\in A_{\crys}$ for a generator of $\bZ_p(1)=(\Fil^1 A_{\crys})^{\varphi=p}$ as usual.
Set $B_{\crys}=A_{\crys}[\frac{1}{t}]$. 
There is a filtration on $K\otimes_{K_0}B_{\crys}$. 
Also, the group $G_K$ acts on $W(R), A_{\crys}$ and $B_{\crys}$ naturally. 

We fix a uniformizer $\pi$ and $p$-power roots $\pi_n$ of $\pi$ in $\overline{K}$ so that $\pi_{n+1}^p=\pi_n$, $\pi_0=\pi$. 
This system gives an element $[\underline{\pi}]\in W(R)\subset A_{\crys}$ via the Teichmuller lift of $(\pi_n~\textnormal{mod}~p)_{n\geq 0}\in R$. 
The ring $B_{\st}$ is a polynomial algebra $B_{\crys}[\log(u)]$ with a variable $\log(u)$, on which the Frobenius endomorphism and the Galois action extend. 
This Galois action depends on $[\underline{\pi}]$ in the following way: $\sigma(\log(u))=\log(u)+\log(\beta(\sigma))$, where $\beta(\sigma)$ is the Teichmuller lift of $p$-power roots of unity defined by $\sigma([\underline{\pi}])=\beta(\sigma)\cdot[\underline{\pi}]$. Note that $\log(\beta(\sigma))$ makes sense in $A_{\crys}$, and in fact it is a generator of $\bZ_p(1)=(\Fil^1 A_{\crys})^{\varphi=p}$. 
Furthermore, it is equipped with a $B_{\crys}$-derivation $N$ such that $N(\log(u))=-1$ and $N\varphi=p\varphi N$. 
There is also a filtration on $K\otimes_{K_0}B_{\st}$ which is compatible with the filtration on $K\otimes_{K_0}B_{\crys}$ under the inclusion $B_{\crys}\subset B_{\st}$. 

\begin{rem}
Our normalization of the monodromy operator $N$ is different from some reference, but this one fits nicely with the geometric theory. 
\end{rem}

We consider a contravariant functor 
$V\mapsto D_{\st}(V)=(B_{\st}\otimes_{\bQ_p}V^{\vee})^{G_K}$ from the category of $\bQ_p$-representations of $G_K$ to the category of filtered $(\varphi, N)$-modules. 
(The structure of a filtered $(\varphi, N)$-module on $D_{\st}(V)$ is induced by the filtration on $B_{\st}\otimes_{K_0}K$, and $(\varphi, N)$ on $B_{\st}$.)

Note that we use the \emph{contravariant} functor, which is more convenient to discuss the integral theory later. 

\begin{defn}\label{Hodge-Tate}
We say that a finite-dimensional $\bQ_p$-representation $V$ of $G_K$ is semistable if
$\dim_{\bQ_p}V=\dim_{K_0}D_{\st}(V)$. 
\end{defn}

If $V$ is semistabe, the filtered $(\varphi, N)$-module $D_{\st}(V)$ satisfies the condition called weak admissibility. 
In fact, it is first proved by ~\cite{Colmez-Fontaine} that $D_{\st}$ induces an anti-equivalence of categories between the category of semistable representations and the category of weakly admissible filtered $(\varphi,N)$-modules. 
There is a canonical identification $V= \Hom_{\Fil,\varphi, N}(D_{\st}(V),B_{\st})$ as representations of $G_K$. 

Also, recall that a semistable representation $V$ is crystalline if and only if $N=0$ on $D_{\st}(V)$.  
If $V$ is crystalline, there is a natural isomorphism of representations of $G_K$
\[
\Hom_{\Fil,\varphi}(D_{\st}(V), B_{\crys})\to\Hom_{\Fil,\varphi,N}(D_{\st}(V), B_{\st})
\]
because $B_{\st}^{N=0}=B_{\crys}$. 

We fix our convention on the Hodge-Tate weights.  
Since $D_{\st}(V)$ is a filtered $(\varphi, N)$-module, $D_K=K\otimes_{K_0}D_{\st}(V)$ has a decreasing filtration.  

\begin{defn}
A semistable representation $V$ has the nonnegative Hodge-Tate weights if $\Fil^0 D_K=D_K$ and $\Fil^{w+1}D_K=0$ for some nonnegative integer $w$, in which case the Hodge-Tate weights are in $\{0,\dots, w\}$. 
\end{defn}

In particular, the cyclotomic character $\bQ_p(1)$ has the Hodge-Tate weight one. 

In the rest of this section, we will consider semistable representations with the nonnegative Hodge-Tate weights unless otherwise mentioned. 

\subsection{Fontaine-Laffaille theory}
There are several integral theories. 
We begin under the Fontaine-Laffaille condition. 
Namely,  we assume $K=K_0$, $V$ is crystalline and $V$ has the Hodge-Tate weight in $\{0,\dots, w\}$ with $w<p-1$. 
We will write $D$ for $D_{\st}(V)$. 

\begin{defn}
A $\varphi$-stable $W$-lattice $M$ in $D$ is called strongly divisible if it satisfies
$\varphi(M\cap \Fil^r D)\subset p^r M$ for any integer $r\geq 0$. 
\end{defn}

\begin{rem}\label{divisibility}
Submodules $M\cap \Fil^r D$ defines a decreasing filtration on $M$, and its graded pieces $\gr^r (M)$ are torsion-free. 
The weak admissibility of $D$ implies that $\sum_{r\geq 0}\frac{\varphi}{p^r}(M\cap\Fil^r D)=M$. 
\end{rem}

We associate a $\bZ_p$-representation $T_{W}(M)$ of $G_K$ with a strongly divisible lattice $M$ by $T_{W}(M)=\Hom_{\Fil,\varphi}(M, A_{\crys})$. 
One can prove that it defines a $G_K$-stable lattice in $V$ via an injection
\[
T_{W}(M)\to \Hom_{\Fil,\varphi, N}(D_{\st}(V), B_{\crys})\overset{\cong}{\to}\Hom_{\Fil,\varphi}(D_{\st}(V), B_{\st})= V. 
\]

Breuil's proof of the following theorem is based on works of Laffaille and Fontaine-Laffaille ~\cite{Fontaine-Laffaille}. 

\begin{thm}[\cite{Breuil1999}*{Proposition 3}]
The contravariant functor $T_{W}$ induces an anti-equivalence of categories between the category of strongly divisible lattices in $D$ and the category of $G_K$-stable lattices in $V$. 
\end{thm}

\subsection{Breuil modules}
Fontaine-Laffaille theory requires a strong assumption. 
Here, we will only assume $w< p-1$. 
Namely, we allow any ramification in $K$ and nonzero monodromy operators on $(\varphi, N)$-modules. 
Still in this case, there is a nice classification by lattices in Breuil modules ~\cite{Liu:Breuil}. 

At the beginning, we do not assume $w< p-1$. 
Let $S$ be the $p$-adic completion of the PD hull of a homomorphism $W(k)[u]\to O_K; u\mapsto \pi$. 
We equip $S$ with a filtration coming from the PD structure, a Frobenius endomorphism $\varphi;u\mapsto u^p$, and a $W(k)$-derivation $N;u\mapsto -u$.  
There is an embedding $S\to A_{\crys};u\mapsto [\underline{\pi}]$. 

\begin{defn}
A Breuil module over $S_{K_0}$ is a filtered $\varphi$-module $\cD$ over $S_{K_0}$ with a monodromy operator $N_{\cD}$ in the following sense:
\begin{itemize}
\item $\cD$ is a finite free $S_{K_0}$-module equipped with a decreasing filtration over $S_{K_0}$ such that $\Fil^0(\cD)=\cD$ and $\Fil^{i}_{S}\cdot\Fil^j(\cD)\subset \Fil^{i+j}{\cD}$.   
\item $\cD$ is given a Frobenius endomorphism $\varphi_\cD\colon\cD\to\cD$ whose determinant is invertible in $S_{K_0}$. 
\item $N_\cD$ satisfies $N_{\cD}(fm)=N_S(f)m+fN_\cD(m)$, $N_\cD\varphi=p\varphi N_\cD$ and $N_\cD(\Fil^i\cD)\subset \Fil^{i-1}\cD$. (The last condition is the Griffiths transversality.)
\end{itemize}
\end{defn}

For the filtered $(\varphi, N)$-module $D$ corresponding to a semistable representation $V$, Breuil ~\cite{Breuil97} defined the structure of a Breuil module on $\cD:=S\otimes_{W}D$. 
There is a natural homomorphism
\[
V=\Hom_{\Fil,\varphi, N}(D_{\st}(V), B_{\st})\to \Hom_{\Fil, \varphi}(\cD,B_{\crys})
\]
induced by the reduction modulo $\log(u)$. 
Combining Breuil's fundamental results~\cite{Breuil97} and ~\cite{Liu:Breuil}*{Lemma 3.4.3}, one can show that this is an isomorphism. 
However, one should be careful for Galois actions as follows.  
Set $K_{\infty}=\bigcup_n K(\pi_n)$ and $G_{\infty}=\Gal(\overline{K}/K_{\infty})$. 
The embedding $S\to A_{\crys}$ is invariant under the action of $G_\infty$. 
Therefore, $G_{\infty}$ acts on $\Hom_{\Fil, \varphi}(\cD,B_{\crys})$ naturally, and the above isomorphism commutes with $G_{\infty}$-actions. 

We introduce two classes of lattices in $\cD$. 

\begin{defn}
\begin{enumerate}
\item A quasi-strongly divisible lattice of weight $\leq w$ is a finite free $\varphi$-stable $S$-submodule $\cM\subset \cD$ such that 
\[
K_0\otimes_W \cM\overset{\cong}{\to}\cD \quad \textnormal{and} \quad \varphi(\Fil^w\cM)\subset p^w\cM, 
\]
where $\Fil^w\cM=\cM\cap \Fil^w\cD$. 
\item A strongly divisible lattice is an $N$-stable quasi-strongly divisible lattice. 
\end{enumerate}
\end{defn}

\begin{rem}
Similar to Remark ~\ref{divisibility}, the weak admissibility of $D$ implies that $\frac{\varphi}{p^w}(\Fil^w \cM)$ generates $\cM$. 
\end{rem}

\begin{exam}
Suppose $K=K_0$, $V$ is crystalline and $w< p-1$. 
We let $M$ be a strongly divisible lattice in $D$. 
Then, $\cM=S\otimes_W M$ is a strongly divisible lattice in $\cD$. 
\end{exam}

We associate a $\bZ_p$-representation $T_{\crys}(\cM)$ of $G_{\infty}$ with a quasi-strongly divisible lattice $\cM$ by $T_{S}(\cM)=\Hom_{\Fil,\varphi}(\cM, A_{\crys})$. 
One can prove that it defines a $G_{\infty}$-stable lattice in $V$ via an injection
\[
T_{S}(\cM)\to \Hom_{\Fil,\varphi}(\cD, B_{\crys})\overset{\cong}{\gets}\Hom_{\Fil,\varphi, N}(D, B_{\st})= V. 
\]
In fact, this lattice is $G_K$-stable if $\cM$ is strongly divisible and we regard it as a representation of $G_K$. 
We remark that the lattice is independent of $\pi$. 

\begin{rem}
Under the Fontaine-Laffaille condition, a natural homomorphism 
$T_{W}(M)\to T_{S}(S\otimes_W M)$ is an isomorphism of representations of $G_K$. 
\end{rem}

Breuil ~\cite{Breuil} conjectured that strongly divisible lattices in $\cD$ classify $G_K$-stable lattices in $V$ if $w< p-1$. 
This is proved by Liu by using Kisin modules, which we discuss next. 

\begin{thm}[\cite{Liu:Breuil}]
Assume $w< p-1$. 
The contravariant functor $T_S$ induces an anti-equivalence of categories between strongly divisible lattices in $\cD$ and $G_K$-stable lattices in $V$. 
\end{thm}

\subsection{Kisin modules}
If $w\geq p-1$, classification of  $G_K$-stable lattices is much harder. 
Fortunately, if we only consider actions of $G_\infty$, there is a good classification by Kisin modules.
 
We need one more ring defined by $\fS=W(k)\llbracket u\rrbracket$. 
It has a Frobenius endomorphism induced by $u\mapsto u^p$. 
Let $E(u)$ be the minimal polynomial of $\pi$ over $W$, which is Eisenstein of degree $e$. 

\begin{defn}
A Kisin module of height $\leq w$ is a pair of a finite free $\fS$-module $\fM$ and a $\varphi$-linear map $\varphi_{\fM}:\fM\to\fM$ such that the cokernel of the linear extension $\fS\otimes_{\varphi,\fS}\fM\to \fM$ of $\varphi_{\fM}$ is killed by $E(u)^w$.  
\end{defn}

\begin{rem}
This is also called a Breuil-Kisin module. 
We call it a Kisin module to distinguish it from a Breuil module. 
\end{rem}

We associate a $\bZ_p$-representation $T_{\fS}(\fM)$ of $G_{\infty}$ with a Kisin module $\fM$ by $T_{\fS}(\fM)=\Hom_{\varphi}(\fM, W(R))$. 
(This is compatible with the standard definition, see the proof of ~\cite{Liu10}*{Lemma 3.1.1}.) 

Let $\cO\subset S_{K_0}$ be the ring of rigid-analytic functions on the rigid-analytic open unit disc ~\cite{Kisin}*{(1.1.1)}. 
It is stable under the Frobenius endomorphism, and contains $\fS$. 

We let $V$ be a semistable representation of $G_K$, $D=D_{\st}(V)$ and $\cD=S\otimes_W D$. 
Kisin proved there is a $\varphi$-stable submodule $\fM_0\subset\cO\otimes_{K_0}D$ which is a Kisin module such that $\cO\otimes_{\fS}\fM_0\cong \cO\otimes_{K_0}D$ and natural homomorphisms
\begin{align*}
T_{\fS}(\fM_0)  
\to\Hom_{\Fil,\varphi}(\cD,B_{\crys})\overset{\cong}{\gets}\Hom_{\Fil,\varphi,N}(D,B_{\st})=V 
\end{align*}
produce a $G_\infty$-stable lattice in $V$. 
(The first homomorphism is well-defined, see below.)
As a Kisin module, $\fM_0$ is unique up to isogeny. 

By the same procedure, any Kisin module of the maximal rank in $\bQ_p\otimes_{\bZ_p}\fM_0$ gives a $G_\infty$-stable lattice in $V$. 

Now, suppose $V$ has the Hodge-Tate weights in $\{0,\dots, w\}$. 
Kisin further proved the following;

\begin{thm}[\cite{Kisin}]
The above construction gives an equivalence of categories between the category of Kisin modules of the maximal rank inside $\bQ_p\otimes_{\bZ_p}\fM_0$ and the category of $G_{\infty}$-stable lattices in $V$. All such Kisin modules have height $\leq w$. 
\end{thm}

When we use the above theorem later in this paper, we implicitly fix $\fM_0$. 
We will denote by $\fM$ the Kisin module attached to a $G_{\infty}$-stable lattice $T$. 
As a Kisin module, the isomorphism class of $\fM$ is well-defined. 
We hope this makes no confusion.  

Let $\fM$ be such a Kisin module. 
Then, $\cM=S\otimes_{\varphi, \fS}\fM$ is a lattice in $\cD$. 
In particular, a filtration on $\cM$ is induced and it can be described by, cf.~\cite{Liu:Breuil}*{Corollary 3.2.3}, 
\[
\Fil^r(\cM)=\bigl\{m\in\mathcal{M}\mid (1\otimes\varphi)(m)\in (\Fil^rS)\otimes_{\fS}\fM\bigr\}. 
\]
Here $\Fil^r S$ is the natural filtration induced by the PD structure.  

One can check easily that it is a quasi-strongly divisible lattice of weight $\leq w$. 
Liu proved that it is strongly-divisible if $T$ is $G_K$-stable and $p>2$ ~\cite{Liu:monodromy}*{Proposition 2.13}. 
Moreover, he proved that $T_{\fS}(\fM)\cong T_S(\cM)$ if $w<p-1$~\cite{Liu:Breuil}*{Lemma 3.3.4}. 

\subsection{Comparison theorems}
We keep the notation. 
In the rational theory, we have the comparison isomorphism
\[
\textnormal{comp}\colon B_{\st}\otimes_{K_0}D\overset{\cong}{\longrightarrow}B_{\st}\otimes_{\bQ_p}V^{\vee}.
\] 
By the reduction modulo $\log(u)$, one obtains 
\[
\overline{\textnormal{comp}}\colon B_{\crys}\otimes_{K_0}D \overset{\cong}{\longrightarrow} B_{\crys}\otimes_{\bQ_p}V^{\vee}. 
\]
We need an integral version of it. 
\
\begin{thm}[Liu]\label{comparison:Kisin}
There is a natural homomorphism
\[A_{\crys}\otimes_{\varphi,\fS}\fM \to A_{\crys}\otimes_{\bZ_p}T^{\vee} 
\]
which preserves $G_{\infty}$-actions, Frobenius endomorphisms and filtrations. 
It admits an inverse up to $t^w$. 
Moreover, the following diagram commutes:
\begin{equation*}
 \begin{split}
 \xymatrix{  B_{\crys} \otimes_{K_0} D \ar[r]^-{\overline{\textnormal{comp}}}&  B_{\crys}\otimes_{\bQ_p}V^{\vee}  \\ 
A_{\crys} \otimes_{\varphi, \fS}\fM \ar[r] \ar[u]& A_{\crys}\otimes_{\bZ_p}T^{\vee}.  \ar[u]}
 \end{split}
\end{equation*}

\end{thm}

\begin{proof}
This follows from ~\cite{Liu:torsion}*{Lemma 5.3.4} and ~\cite{Liu:torsion}*{Theorem 5.4.2} except two points. 
One is that it excludes partially the case $p=2$ because of ~\cite{Liu:torsion}*{Assumption 5.2.1}, but this assumption is not used in their proofs. 
See also ~\cite{Liu10}*{Proposition 3.1.3}. 

The other point is preservation of filtrations, see ~\cite{Liu:torsion}*{Remark 5.4.3}. 
More precisely, it is easy to see that the homomorphism in the statement preserves filtrations after inverting $p$. 
By taking intersections, one can recover filtrations of each sides of the homomorphism.  Therefore it also respects filtrations. 
\end{proof}

\begin{rem}
As a submodule of $B_{\crys}\otimes_{K_0}D$, $A_{\crys}\otimes_{\varphi,\fS}\fM$ is $G_K$-stable, and the homomorphism in the theorem respects such a $G_K$-action. 
Indeed, one can describe the $G_{K}$-action on
$A_{\crys}\otimes_{\varphi,\fS}\fM$, as a subset of $A_{\crys}\otimes_{S}\cD$, in the following way. 
Let $\sigma \in G_K$ and $\beta (\sigma)=\frac{\sigma([\underline{\pi}])}{[\underline{\pi}]}\in \bZ_p(1)\subset A_{\crys}$, then the action of $\sigma$ is given by 
\[
\sigma(a\otimes x)=\sum^{\infty}_{j=0} \frac{\log(\beta(\sigma))^j}{j!}\sigma(a)\otimes N^j(x). 
\] 
This follows from ~\cite{Liu10}*{(3.2.1)}. 
Note that there is no sign in the formula above because of our normalization of $N$ on $B_{\st}$. 
\end{rem}

\subsection{A lattice in $D_{\dR}$}
Recall the period ring $B_{\dR}$, which is independent of $\pi$. 
It contains $K\otimes_{K_0}B_{\crys}$, and there is also a natural embedding $K\otimes_{K_0}B_{\st}\to B_{\dR};\log(u)\mapsto\log(\frac{[\underline{\pi}]}{\pi})$. 
It is equipped with a filtration, and filtrations on $K\otimes_{K_0}B_{\crys}$ and $K\otimes_{K_0}B_{\crys}$ come from that of $B_{\dR}$. 

We define a filtered $K$-vector space $D_{\dR}$ by
\[
D_{\dR}=D_{\dR}(V)=(B_{\dR}\otimes_{\bQ_p}V^{\vee})^{G_K}. 
\] 
The embedding $B_{\st}\to B_{\dR}$ induces $D_K\overset{\cong}{\to}D_{\dR}$. 
In our notation $D$ depends on $\pi$, but $D_{\dR}$ is independent of $\pi$.   

Given a $G_K$-stable lattice $T$, we will produce a $O_K$-lattice in $D_{\dR}$. 
Its properties are studied by Liu. 

Specializing by $\fS\mapsto O_K; u\to \pi$, we identify $D_K$ and $O_K\otimes_{S}\cD$. 

\begin{defn}\label{lattice}
We define a $O_K$-lattice $M_{\dR}(T)$ in $D_{\dR}$ by the image of a composition
\[
\fS\otimes_{\varphi, \fS}\fM\overset{1\otimes\varphi}{\longrightarrow} O_K\otimes_{S}\cD=D_K\cong D_{\dR}. 
\]
The image of the filtration of $\cM$ defines a filtration on $M_{\dR}(T)$. (Note that the image of $\cM$ equals $M_{\dR}(T)$.)
\end{defn} 

\begin{rem}
There is a natural homomorphism $\gr^r M_{\dR}(T)\to \gr^r D_{\dR}$ whose image is a $O_K$-lattice. We use these lattices later. 
\end{rem}

The filtered lattice $M_{\dR}(T)$, \textit{a priori}, depends on choices of $\pi_n$. 
According to Liu, it turns out that it is independent of such choices. 

\begin{thm}[Liu]\label{independence of pi}
The filtered lattice $M_{\dR}(T)$ is independent of choices of $\pi_n$. 
\end{thm}

\begin{proof}
As a $O_K$-lattice, this is ~\cite{Liu:compatibility}*{Proposition 4.2.1}. 
(Note that, in the proof of \textit{loc.~cit.}, $\Fil^1\cM$ means $\Fil^1S\cdot\cM$ and is different from ours.)
His argument can be applied to the filtration by using the following facts. 
First, the image of homomorphism
\[
\Fil^r(A_{\crys}\otimes_S \cM)\to\gr^0 A_{\crys}\otimes_{O_K} D_K=\widehat{O}_{\overline{K}}\otimes_{O_K} D_K
\]
equals the image of 
$\widehat{O}_{\overline{K}}\otimes_{O_K}(O_K\otimes_S \Fil^r\cM)$. 
The other fact is under the identification $A_{\crys}\otimes_{\varphi,\fS}\fM=A_{\crys}\otimes_{\varphi,\fS'}\fM'$ in \textit{loc.~cit.} filtrations coincide. 
(To define the right hand side, we use a different uniformizer $\pi'$ and its $p$-power roots.)
To verify the latter fact, one may invert $p$. 
Then, we claim that the filtration on $A_{\crys}\otimes_S\cD$ is induced by the filtration on $B_{\crys}\otimes_{S_{K_0}}\cD\cong B_{\crys}\otimes_{\bZ_p}V^{\vee}$. 
Indeed, it is easy to see that multiplication by $t$ on $A_{\crys}\otimes_S\cD$ is strictly compatible with the filtration, and the claim follows. 
The claim implies that the filtration is independent of choices of $\pi_n$. 
\end{proof}

\begin{exam}
For the cyclotomic character $\bZ_p(1)\subset \bQ_p(1)$, 
\[
M_{\dR}(\bZ_p(1))=O_K\subset K=D_{\dR}(\bQ_p(1))
\]
with nonzero $\gr^1$ ~\cite{Liu:torsion}*{Example 2.3.5}.  
In general,  one can check that
\[
M_{\dR}(T\otimes_{\bZ_p}\bZ_p(1))=M_{\dR}(T)\otimes_{O_K} M_{\dR}(\bZ_p(1))\subset D_{\dR}\otimes_K D_{\dR}(\bQ_p(1)).
\]
\end{exam}

\begin{rem}
The example above allow one to define $M_{\dR}(T)$ for a $G_K$-stable lattice in a semistable representation $V$ with any Hodge-Tate weights. 
\end{rem}

\begin{prop}\label{Fontaine-Laffaille}
Suppose $K=K_0$, $V$ is crystalline and $w\leq p-2$. 
We write $M$ for the strongly divisible lattice corresponding to $T$. 
Then, $M=M_{\dR}(T)$ as filtered lattices.  
\end{prop}

\begin{proof}
This is seen in the proof of ~\cite{Liu:compatibility}*{Proposition 4.1.2.(6)}. 
\end{proof}

Let $K'$ be a finite extension of $K$ inside $\overline{K}$. 
The period ring $B_{\dR}$ is independent of $K$ and $K'$. 
We can identify naturally $D_{\dR}(V)_{K'}$ and $D_{\dR}(V|_{G_{K'}})$. 
\begin{prop}\label{base change}
Under the identification above, two filtered $O_{K'}$-lattices $O_{K'}\otimes M_{\dR}(T)$ and $M_{\dR}(T|_{G_{K'}})$ coincide. 
\end{prop}

\begin{proof}
Like the proof of Proposition ~\ref{independence of pi}, 
this can be checked by modifying the proof of ~\cite{Liu:compatibility}*{Proposition 4.2.2.(2)}. 
\end{proof}

\section{External Integral $p$-adic Hodge Theory}
We review geometric aspects of integral $p$-adic Hodge theory. 
We keep the notation as before. 

\subsection{Rational comparison}
Let $X$ be a proper semistable scheme over $O_K$. 
We write $X^\times$ for $X$ with the log structure given by the special fiber. 
We have the $p$-adic \'etale cohomology $H^i_{\et}(X_{\overline{K}},\bQ_p)$, the de Rham cohomology $H^i_{\dR}(X_K)$, and the log crystalline cohomology $H^i_{\crys}(X^\times_k/W)$ relative to $W$ with a log structure attached to $\bN\to W;1\mapsto 0$.

The semistable conjecture, which is first proved by Tsuji \cite{Tsuji}, says that there is a comparison isomorphism
\[
\textnormal{comp}\colon
B_{\st}\otimes_{K_0}H^i_{\crys}(X^\times_k/W)\overset{\cong}{\longrightarrow}B_{\st}\otimes_{\bQ_p}H^i_{\et}(X_{\overline{K}},\bQ_p). 
\]
Such a period morphism is essentially unique ~\cite{Niziol}. 
By the reduction modulo $\log(u)$, it induces
\[
\overline{\textnormal{comp}}\colon
B_{\crys}\otimes_{K_0}H^i_{\crys}(X^\times_k/W)\overset{\cong}{\longrightarrow}
B_{\crys}\otimes_{\bQ_p}H^i_{\et}(X_{\overline{K}},\bQ_p). 
\]

On the other hand, we have the Hyodo-Kato isomorphim
\[
K\otimes_W H^i_{\crys}(X^\times_k/W)\cong H^i_{\dR}(X_K).
\] 
Therefore, setting
\[
D=D_{\st}(H^i_{\et}(X_{\overline{K}},\bQ_p)^{\vee}) \quad \textnormal{and} \quad D_{\dR}=D_{\dR}(H^i_{\et}(X_{\overline{K}},\bQ_p)^{\vee}), 
\]
we obtain isomorphisms compatible with additional structures
\[
D\cong K_0\otimes_W H^i_{\crys}(X^\times_k/W) \quad \textnormal{and} \quad
D_K\cong D_{\dR}\cong H^i_{\dR}(X_K).
\]

Note that a comparison isomorphism appeared in the de Rham conjecture, which is obtained by tensoring $B_{\dR}$,  
\[
B_{\dR}\otimes \textnormal{comp}\colon
B_{\dR}\otimes_K H^i_{\dR}(X_K)\overset{\cong}{\to}B_{\dR}\otimes_{\bQ_p}H^i_{\et}(X_{\overline{K}},\bQ_p)
\]
is independent of $\pi$. 

By taking graded pieces, it induces the Hodge-Tate decomposition
\[
\gr^0(B_{\dR}\otimes \textnormal{comp})\colon
\bigoplus_r \widehat{O}_{\overline{K}}(-r)\otimes_{O_K} \gr^r H^i_{\dR}(X_K)\overset{\cong}{\to} \widehat{O}_{\overline{K}}\otimes_{\bZ_p}H^i_{\et}(X_{\overline{K}},\bQ_p). 
\]

Finally, recall that the de Rham conjecture and the Hodge-Tate decomposition are true for any proper smooth scheme over $K$, and the period morphisms are compatible with finite base extensions. 

\subsection{Integral comparison}
Now we consider the $p$-adic \'etale cohomology $H^i_{\et}(X_{\overline{K}},\bZ_p)$. 
We also have the log crystalline cohomology $H^i_{\crys}(X^\times/S)$ relative to $S$ with a log structure attached to $\bN\to S;1\mapsto u$. 

First, we mention a variant of the Hyodo-Kato isomorphism. 

\begin{prop}[\cite{Hyodo-Kato}*{(5.2)}]
There is a natural isomorphism 
\[
K_0\otimes_W H^i_{\crys}(X^\times/S)\cong S_{K_0}\otimes_W H^i_{\crys}(X^\times_k/W)
\]
compatible with Frobenius endomorphisms and monodromy operators. 
\end{prop}

See also ~\cite{Beilinson}*{1.13--16} for another proof based on Dwork's trick.  
In particular, $K_0\otimes_W H^i_{\crys}(X^{\times}/S)$ is free over $S_{K_0}$ and has the structure of a Breuil module. 
(The other conditions are easier to check.)

The integral comparison is the following: 

\begin{thm}[\cite{Faltings02}]\label{weak}
There is a natural homomorphism 
\[A_{\crys}\otimes_S H^i_{\crys}(X^\times/S) \to A_{\crys}\otimes_{\bZ_p} H^i_{\et}(X_{\overline{K}},\bZ_p), 
\]
which preserves $G_{\infty}$-actions, Frobenius endomorphisms, and filtrations.  
It admits an inverse up to $t^{d}$, where $d$ is the relative dimension of $X/O_K$. 
Moreover, the following diagram commutes
\begin{equation*}
 \begin{split}
 \xymatrix{  B_{\crys} \otimes_{K_0} H^i_{\crys}(X^\times_k/W)\ar[r]^-{\overline{\textnormal{comp}}} &  B_{\crys}\otimes_{\bQ_p}H^i_{\et}(X_{\overline{K}},\bQ_p)  \\ 
A_{\crys} \otimes_{S}H^i_{\crys}(X^\times/S) \ar[r] \ar[u]& A_{\crys}\otimes_{\bZ_p}H^i_{\et}(X_{\overline{K}},\bZ_p).  \ar[u]}
 \end{split}
\end{equation*}
\end{thm}

\begin{rem}
Faltings proved a more general result. 
The theorem above is a combination of statements in ~\cite{Faltings02}*{5.4 and Remarks}. 

Another proof can be found in ~\cite{Bhatt}*{10.2}. 
In \textit{ibid}, semistable means strictly semistable, but his proof works for the semistable case. 
Also, geometrically connectedness assumed in ~\cite{Bhatt}*{Theorem 10.17} can be removed by \'etale base change. 

Bhatt's construction of the period morphism is based on Beilinson's idea of using $h$-topology. 
In ~\cite{Beilinson}, Beilinson himself also constructed essentially the same period morphism, though the integral comparison is not discussed there. 

Uniqueness of integral period morphisms seems subtle, but it is not necessary for our purpose once we fix a choice of constructions. 
\end{rem}

\begin{rem}\label{full action:geometric}
As in the case of Kisin modules, 
$A_{\crys} \otimes_{S}H^i_{\crys}(X^\times/S)$ is a $G_K$-stable submodule of 
$B_{\crys} \otimes_{K_0} H^i_{\crys}(X^\times_k/W)$, and the homomorphism in the theorem respects $G_K$-actions. 

Again, one can describe the $G_{K}$-action on $A_{\crys}\otimes_S H^i_{\crys}(X^\times/S)$ as follows, cf. ~\cite{Faltings02}*{p.259}. 
Let $\sigma \in G_K$, then the action of $\sigma$ is given by 
\[
\sigma(a\otimes x)=\sum^{\infty}_{j=0} \frac{\log(\beta(\sigma))^j}{j!}\sigma(a)\otimes N^j(x). 
\] 
\end{rem}

\subsection{Strong comparison}
One might expect that $H^i_{\crys}(X^\times/S)$ modulo torsion is a (quasi-)strongly divisble lattice in the Breuil module $K_0\otimes_W H^i_{\crys}(X^\times/S)$ and it comes from the Kisin module attached to $H^i_{\et}(X_{\overline{K}},\bZ_p)^{\vee}$ modulo torsion. 
(This is a guess in the line of ~\cite{Breuil}*{Question 4.1}, see also ~\cite{Liu:Breuil}*{Remark 4.3.5(2)}.)

To the best of our knowledge, it is not known in general. 
However, if we assume the log Hodge cohomology is torsion-free, we can prove such a statement. 

\begin{thm}[Faltings]\label{strong}
Assume the log Hodge cohomology 
$H^{i-r}(X^\times,\Omega^r)$ are torsion-free for all $i$ and $r$, and $d\leq p-2$. 
Then, the following holds:
\begin{enumerate}
\item $H^i_{\et}(X_{\overline{K}},\bZ_p)$ is torsion-free. 
\item $H^i_{\crys}(X^\times/S)$ is a strongly divisible lattice in $S_{K_0}\otimes_W H^i_{\crys}(X^\times_k /W)$ and $H^i_{\et}(X_{\overline{K}},\bZ_p)^{\vee}=T_{S}(H^i_{\crys}(X^\times/S))$. 
\item $H^i_{\dR}(X^\times)=M_{\dR}(H^i(X_{\overline{K}},\bZ_p)^{\vee})$. Here, $H^i_{\dR}(X^\times)\subset H^i_{\dR}(X_K)$ is the log de Rham cohomology. (Note that it is torsion-free.) 
\end{enumerate}
\end{thm}

\begin{proof}
Our main reference is ~\cite{Faltings99}. 
We remark that his ring $R_V$ is slightly different from our $S$, but his argument works for $S$ as well. We also refer to ~\cite{Faltings02} for some complements. 

If $X$ is smooth, the statement follows from ~\cite{Faltings99}*{Theorem 6}. 
Note that the Hodge spectral sequence degenerates because it does over $K$. 

The semistable case was mentioned as a conjecture in \cite{Faltings99}*{p.132}, but Theorem ~\ref{weak} is now known and the semistable case follows. 

Let us mention more details. 
By ~\cite{Faltings99}*{Thoerem 1*}, $H^i_{\crys}(X^\times /S)$ is free. 
Then, one can easily check that $H^i_{\crys}(X^\times /S)$ is a strongly divisible lattice. 
To show (1) and (3), we identify 
\[
T_S(H^i_{\crys}(X^\times/S))=\Hom_{S,\Fil,\varphi}(H^i_{\crys}(X^\times/S), A_{\crys})
\] with
\[
\Hom_{A_{\crys},\Fil,\varphi}(A_{\crys}\otimes_S H^i_{\crys}(X^\times/S), A_{\crys}). 
\]
On the other hand, we have
\[
H^i_{\et}(X_{\overline{K}},\bZ_p)^{\vee}=
\Hom_{A_{\crys},\Fil,\varphi}(A_{\crys}\otimes_{\bZ_p}H^i_{\et}(X_{\overline{K}},\bZ_p), A_{\crys}) 
\]
and
\[
\Hom_{A_{\crys},\Fil,\varphi}(t^d A_{\crys}\otimes_{\bZ_p}H^i_{\et}(X_{\overline{K}},\bZ_p), A_{\crys})
=H^i_{\et}(X_{\overline{K}}, \bZ_p)^{\vee}. 
\]
The second equality holds because of the assumption $d\leq p-2$. 
Then, we get homomorphisms
\[
H^i_{\et}(X_{\overline{K}},\bZ_p)^{\vee}\to T_{S}(H^i_{\crys} (X^\times/S))
\]
and
\[
T_{S}(H^i_{\crys} (X^\times/S))\to H^i_{\et}(X_{\overline{K}}, \bZ_p)^{\vee},  
\]
which are inverse to each other. 

The identification $K_0\otimes_W H^i_{\crys}(X^\times/S)\cong S_{K_0}\otimes_W H^i_{\crys}(X^\times_k /W)$ specializes to the Hyodo-Kato isomorphism $H^i_{\dR}(X_K)\cong K\otimes_{W}H^i_{\crys}(X^\times/W)$ via base chage to $K$. 
Similarly, $H^i_{\crys}(X^\times/S)$ specializes to the lattice $H^i_{\dR}(X^\times)$ in $H^i_{\dR}(X_K)$ because of freeness. 
Therefore, the last part follows.  
\end{proof}

\subsection{Log smooth reduction}
We need the integral comparison in a more general situation, allowing finite base extensions.

Recall the following result of H. Yoshioka, which says a vertical log smooth scheme over $O_K$ has potentially semistable reduction. 
A proof can be found in Saito's paper ~\cite{Saito:log}. 

\begin{thm}[Yoshioka, ~\cite{Saito:log}*{Theorem 2.9}]\label{potentially semistable}
Let $X^{\times}$ be a fine saturated log scheme which is vertical and log smooth over $O_K$ with the canonical log structure. 
Then, there exists a finite extension $L$ of $K$ such that, for any finite extension $K'$of $L$,  the base change $X^{\times}_{O_{K'}}$ (as a fine saturated log scheme) becomes semistable after log blow-ups. 
\end{thm}

\begin{cor}\label{log smooth}
Let $X^{\times}$ be a proper, vertical and log smooth scheme over $O_K$. 
There exists a finite extension $L$ of $K$ such that, for any finite extension $K'$ of $L$, the semistable conjecture, Theorem~\ref{weak} and Theorem~\ref{strong} holds for $X^{\times}_{O_{K'}}$. 
\end{cor}

\begin{proof}
We use Theorem ~\ref{potentially semistable}. 
Since the generic fiber does not change under log blow-ups, the \'etale cohomology are invariant under log blow-ups. 
It is also known that log blow-ups preserve the log crystalline cohomology, see the proof of ~\cite{Niziol:comparison}*{Proposition 2.3}. 
Similarly, the log Hodge cohomology and the log de Rham cohomology are also preserved under log blow-ups.

Therefore, by the case of semistable reduction, the statement is true with possibly non-canonical period morphisms. 

Now, we explain why these period morphisms are canonical. 
The point is that, by construction, period morphisms for semistable models factor through certain cohomology whch can be defined for log schemes over $O_K$ and is compatible with finite base extensions.  
(For instance, Beilinson ~\cite{Beilinson} used the absolute log crystalline cohomology of $X^{\times}_{O_{\overline{K}}}$, and Bhatt~\cite{Bhatt} used a log derived de Rham version of it.)
So, it suffices to observe that, given log blow-ups $X_1\to X^{\times}_{O_{K'}}$ and $X_2\to X^{\times}_{O_{K'}}$ from semistable schemes, one can find an extension $L'$ of $K'$, and a proper semistale model $X_3$ of $X^{\times}_{L'}$ which dominates $X_1$ and $X_2$. 
This can be seen from Theorem~\ref{potentially semistable}.   
\end{proof}

\section{Heights of Pure Motives}
For a pure motive over a number field, Kato ~\cite{Kato1} defined its height which reflects arithmetic degenerations. 
This generalizes the Faltings height of abelian varieties ~\cite{Faltings83}. 
Roughly speaking, he compared a lattice in the Betti realization and a lattice in the de Rham realization. 

Assuming $M$ has semistable reduction everywhere, we modify his definition by using a (possibly) different lattice $M_{\dR}(T)$ in the de Rham realization. 

\subsection{Pure motives with integral coefficients}

We work over a number field $F$, and $O_F$ denotes the ring of integers in $F$. 

By a pure $\bQ$-motive we mean a Grothendieck motive. 
Namely, we use the homological equivalence with respect to, say, de Rham realizations. 

Let $M_{\bQ}$ be a pure $\bQ$-motive of weight $w$ over $F$. 
Then, the adelic \'etale realization $M_{\bA^f_{\bQ}}$ is a $\bA^f_{\bQ}$-module with a $G_F$-action. 
A free $\hat{\bZ}$-submodule $T$ of $M_{\bA^f_\bQ}$ is called a lattice if $\bA^f_{\bQ}\otimes_{\hat{\bZ}}T=M_{\bA^f_\bQ}$. 

\begin{defn}
A pure $\bZ$-motive $M=(M_\bQ, T)$ over $F$ of weight $w$ is a pair of a pure $\bQ$-motive $M_{\bQ}$ of weight $w$ over $F$ and a $G_F$-stable lattice $T$ in its adelic realization. 
\end{defn}

Our main focus is a motive which has semistable reduction everywhere. 
A precise definition in this paper is as follows:

\begin{defn}
Let $v$ be a finite place of $F$ above a prime number $p$. 
\begin{enumerate}
\item $M_{\bQ}$ (or $M$) has good reduction at $v$ if the following conditions hold:
\begin{itemize}
\item The $\ell$-adic realization $M_\ell$ is unramified at $v$ for any $\ell\neq p$.  
\item The $p$-adic realization $M_p$ is crystalline at $v$. 
\end{itemize}
\item $M_\bQ$ (or $M$) has semistable reduction at $v$ if the following conditions hold:
\begin{itemize}
\item The inertia $I_v$ acts unipotently on $M_\ell$ for any $\ell\neq p$.
\item The $p$-adic realization $M_p$ is semistable at $v$. 
\end{itemize}
\end{enumerate}
\end{defn}

\begin{rem}
A pure $\bZ$-motive $M$ becomes semistable everywhere in our sense after a finite base extension since de Rham representations are potentially semistable by works of Berger, Andre, Kedlaya and Mebkhout. 
(Also use Grothendieck's monodromy theorem at $\ell\neq p$.)
\end{rem}

We denote by $M_{\dR}$ the de Rham realization of $M_\bQ$, which is a $F$-vector space with the Hodge filtration. 

Assume $M$ has semistable reduction everywhere. 
We will construct a lattice $M_{\dR}(T^{\vee})$ in $M_{\dR}$ from the lattice $T$, which we will call the de Rham realization of $M$. 

We let $v$ be a finite place of $F$ above a prime number $p$. 
Recall that $D_{\dR}(M_p^{\vee})\cong F_v\otimes_F M_{\dR}$. 
The lattice $T$ induces a $G_F$-stable lattice $T_p$ in $M_p$. 
Therefore, by the internal integral $p$-adic Hodge theory (Definition ~\ref{lattice}), we obtain a filtered $O_{F_v}$-lattice in $F_v\otimes_F M_{\dR}$, which we denote by $M_{\dR}(T^{\vee})_v$. 
Each graded piece $\gr^r (M_{\dR}(T^{\vee})_v)$ may have torsion, but its free part 
$\gr^r (M_{\dR}(T^{\vee})_v)_{\free}$ is a $O_{F_v}$-lattice in $F_v\otimes_F \gr^r M_{\dR}$. 
We write $\gr^r(M_{\dR}(T^{\vee})_v)_{\torsion}$ for its torsion part. 
For $p$ big enough, $\gr^r (M_{\dR}(T^{\vee})_v)$ is torsion-free by Proposition \ref{Fontaine-Laffaille}, that is, $\gr^r(M_{\dR}(T^{\vee})_v)_{\torsion}=0$. 

By varying $v$, they define a filtered $O_F$-lattice in $M_{\dR}$ and an $O_F$-lattice $\gr^r(M_{\dR}(T^{\vee}))_{\free}$ in $\gr^r(M_{\dR})$. 
This is because the Hodge cohomology of a proper flat model over $O_F$ of a projective smooth variety gives a lattice in the Hodge cohomology of the generic fiber and the assumption of Theorem ~\ref{strong} holds for such a model at almost all finite places. 
(Say, consider places at which the model is smooth.)

\subsection{A definition of the height}

Following Kato, set
\[
L(M) =\bigotimes_{r\in \bZ}(\det \gr^r(M_{\dR}(T^{\vee}))_{\free})^{\otimes r}. 
\]
(Tensor products and determinants are taken over $O_F$, and $\det 0$ is understood as $O_F$.)

Fix an embedding $v$ of $F$ into $\bC$. 
We use the Hodge structure to give a Hermitian metric on $L(M)_v=L(M)\otimes_{O_F}\bC$. 
Kato ~\cite{Kato1}*{1.3} observed that there are natural isomorphisms
\begin{align*}
& L(M)_v\otimes \overline{L}(M)_{v} \\
&\cong(\bigotimes_{r\in \bZ}(\det H^{r,w-r}(M_{\dR}))^{\otimes r})
\otimes(\bigotimes_{r\in \bZ}(\det H^{w-r,r}(M_{\dR}))^{\otimes r}) \\
&\cong \bigotimes_{r\in \bZ}(\det H^{r,w-r}(M_{\dR}))^{\otimes w}
\cong (\det M_{\dR}(T^{\vee}))^{\otimes w}\otimes_{O_F} \bC.
\end{align*}
Here $\overline{L}(M)_v$ means complex conjugate of $L(M)_v$, 
and $H^{r,w-r}$ is the $(r,w-r)$-component of the Hodge decomposition for $v$. 
Note that $r$ times tensor products plays an important role here. 
 
On the other hand, the Betti-\'{e}tale comparison produces a lattice in the Betti realization from $T$, and the Betti-de Rham comparison gives a $\bZ$-structure on $(\det (M_{\dR}(T^{\vee})))^{\otimes w}\otimes_{O_F} \bC$. 
Choose a generator $s_{B,v}$ of this free $\bZ$-module of rank $1$. 
Given $s\in L(M)_v$, write the image of $s\otimes \overline{s}$ as $zs_{B,v}$ with $z\in \bC$. 
Then, we define a metric by the formula
\[
\lvert s\rvert_v=\frac{1}{(\sqrt{2\pi})^{\frac{nw^2}{2}}}\lvert z\rvert ^{1/2}, 
\]
where $n$ is the dimension of $M_{\dR}$. 
This is clearly independent of $s_{B,v}$. 

To deal with the torsion, we introduce
\[
\#L(M)_{\torsion}=\prod_{v:\textnormal{finite}} \prod_r \#(\gr^r(M_{\dR}(T^{\vee})_v)_{\torsion})^r. 
\]

\begin{defn}
Assumptions are as above. 
Take an element $s$ in $L(M)$. 
We define the multiplicative height $H(M)\in \bR_{>0}$ of $M$ by 
\[
H(M)^{[F:\bQ]}=\#L(M)_{\torsion}\cdot \#(L(M)/O_F\cdot s)\cdot \prod_{v\in\Hom(F,\bC)}\lvert s\rvert_v^{-1}. 
\] 
We call $h(M)=\log H(M)$ the (logarithmic) height of $M$. 
(Both heights are independent of $s$.)
\end{defn}

\begin{rem}
We put the factor $\#L(M)_{\torsion}$ to make a behavior of heights under an isogeny better, cf.~Proposition~\ref{formula}. 
Contributions from odd finite places can be controlled by ~\cite{Liu:filtration}*{Lemma 2.4.3}. 

If the Hodge-Tate numbers of $M_p$ are nonpositive, the logarithmic height can be defined as the Arakelov degree of an arithmetic line bundle with $L(M)$ replaced by a possibly larger lattice 
\[
\bigotimes_{r\geq 1}\det \Fil^r(M_{\dR}(T^{\vee}))\subset \bigotimes_{r\geq 1}\det (\Fil^r(M_{\dR}(T^{\vee})))_F\cong L(M)_F.
\]
\end{rem}

\begin{rem}
Our definition of the height essentially follows that of Kato ~\cite{Kato1}*{3.4}, but we use a different lattice in the de Rham realization. 
As shown in the main theorems of this paper, it enable us to relate it to geometry and study its behaviors under \emph{any} isogeny. 

We do not know how to compare our lattice and a lattice used by Kato (beyond Fontaine-Laffialle case). 
\end{rem}

\begin{exam}
Let $\bZ(-1)$ be the Lefschetz motive.  
Our normalization gives $h(\bZ(-1))=0$. 
Indeed, for any Tate twist, $h(\bZ(-\frac{w}{2}))=0$. 

More generally, if the determinant of $M$ is isomorphic to an Artin-Tate motive, which is expected to be true, then $h(M)=h(M(-1))$. 
\end{exam}

\begin{prop}\label{base change:height}
Let $F'$ be a finite extension of $F$. 
We write $M|_{F'}$ for the base change of $M$ to $F'$. 
Then, the equality $h(M|_{F'})=h(M)$ holds. 
\end{prop}

\begin{proof}
This follows from Proposition ~\ref{base change} and definitions. 
\end{proof}

\begin{rem}
Since any pure $\bZ$-motive $M$ becomes semistable in our sense after a finite base extension, we can define the stable height of $M$ by base change to such an extension. 
This is well-defined because of the previous proposition. 
\end{rem}

\subsection{A geometric interpretation}
Let $X$ be a proper scheme of relative dimension $d$ over $O_F$ with semistable reduction everywhere. 
Assume that the log Hodge cohomology over $O_{F_v}$ is torsion-free for any finite place $v$ of $F$.  

We write $M$ for the pure $\bZ$-motive of weight $i$ associated with the $i$-th cohomology of $X$ by considering the free part $H^i_{\et}(X_{\overline{F}},\bZ_p)_{\free}$ of
$H^i_{\et}(X_{\overline{F}},\bZ_p)$. 
So, we have its height $h(M)$. 
(Technically speaking, since $X_F$ is only proper, one needs to use Chow's lemma and resolve singularities to construct the associated motive.)

On the other hand, the log Hodge cohomology defines a $O_F$-lattice in $L(M)_{\bQ}$, and an associated arithmetic line bundle $\cL$. 
We denote by $\widehat{\deg} \cL$ its Arakelov degree, which is defined in the same way as $h(M)$ without the factor $\#L(M)_{\torsion}$.  

\begin{prop}\label{Key}
The difference $\lvert h(M)-\widehat{\deg} \cL\rvert$ is bounded by a constant only depending on $d$ and the Hodge numbers $h^{r,i-r}$. 
\end{prop}

\begin{proof}
The same metric is used to define $h(M)$ and $\widehat{\deg} \cL$. 
So, the torsion part and the size of lattices only matter. 

By Theorem~\ref{strong}, any finite place of $F$ above $p\geq d+2$ do not contribute to the difference. 

Suppose $v$ is a finite place $F$ above $p \leq d+1$, and set $K=F_v$.
We first discuss consequences of comparison theorems. 
As $H^i_{\crys}(X^\times/S)$ is torsion-free, 
Theorem~\ref{weak} gives a homomorphism 
\[
A_{\crys}\otimes_S H^i_{\crys}(X^\times/S) \to A_{\crys}\otimes_{\bZ_p} H^i_{\et}(X_{\overline{K}},\bZ_p)_{\free}, 
\]
which admits an inverse up to $t^d$. 
For any $r$ between $0$ and $i$, consider 
\[
(\gr^r (A_{\crys}\otimes_S H^i_{\crys}(X^\times/S)))^{G_K} \to 
(\gr^r (A_{\crys}\otimes_{\bZ_p} H^i_{\et}(X_{\overline{K}},\bZ_p)_{\free}))^{G_K}.  
\]

The right hand side is equal to $(\gr^r A_{\crys}\otimes_{\bZ_p}H^i_{\et}(X_{\overline{K}},\bZ_p)_{\free})^{G_K}$. 
Under the Hodge-Tate decomposition, this gives a lattice in $H^{i-r}(X_K,\Omega^r)$. 
Next, we claim that the left hand side is equal to $H^{i-r}(X^\times, \Omega^r)$.
It is easy to see that the left hand side coincides with $G_K$-invariants of the image of $\Fil^0 A_{\crys}\otimes_S \Fil^r H^i_{\crys}(X^\times/S)$. 
This image is equal to
\[
\gr^0 A_{\crys}\otimes_{S}\gr^r H^i_{\crys}(X^\times/S)
=\gr^0 A_{\crys}\otimes_{O_K}(O_K \otimes_{S}\gr^r H^i_{\crys}(X^\times/S)). 
\] 
By freeness and base change, 
\[
O_K \otimes_{S}\gr^r H^i_{\crys}(X^\times/S)
=\gr^r H^i_{\dR}(X^\times)
=H^{i-r}(X^\times,\Omega^r). 
\]
Then, we use $(\gr^0 A_{\crys})^{G_K}=O_K$ to conclude. 
(To conclude, we also implicitly use the identification induced by the Hodge-Tate decomposition.)

Thus, one gets an injection
\[
H^{i-r} (X^\times, \Omega^r)\to 
(\gr^r A_{\crys}\otimes_{\bZ_p} H^i(X_{\overline{K}},\bZ_p)_{\free})^{G_K}. 
\]
One can do a similar argument to obtain
\begin{align*}
&(\gr^d A_{\crys}(-d))^{G_K}\otimes_{O_K}H^{i-r}(X^\times,\Omega^r)\\
&\to (\gr^{r+d} A_{\crys}(-d)\otimes_{\bZ_p} H^i_{\et}(X_{\overline{K}},\bZ_p)_{\free})^{G_K}. 
\end{align*}
Existence of an inverse up to $t^d$ implies that
its image contains 
\[
(\gr^r A_{\crys}\otimes_{\bZ_p} H^i_{\et}(X_{\overline{K}},\bZ_p)_{\free})^{G_K}, 
\; \textnormal{regarding} \;
\gr^r A_{\crys}\subset \gr^{r+d} A_{\crys}(-d).
\]
 
We denote the Kisin module attached to the dual of $H^i_{\et}(X_{\overline{K}},\bZ_p)_{\free}$ by $\fM$. 
Then, by Theorem ~\ref{comparison:Kisin},  we have a natural homomorphism
\[A_{\crys}\otimes_{\varphi,\fS}\fM \to A_{\crys}\otimes_{\bZ_p} 
H^i_{\et}(X_{\overline{K}},\bZ_p)_{\free}, 
\]
which admits an inverse up to $t^i$. 

By similar arguments as before, 
we obtain
\[
\gr^r M_{\dR}(H^i_{\et}(X_{\overline{K}},\bZ_p)_{\free})\to 
(\gr^r A_{\crys}\otimes_{\bZ_p} H^i_{\et}(X_{\overline{K}},\bZ_p)_{\free})^{G_K}
\]
and
\begin{align*}
&(\gr^i A_{\crys}(-i))^{G_K}\otimes_{O_K}\gr^r M_{\dR}(H^i_{\et}(X_{\overline{K}},\bZ_p)_{\free})\\
&\to (\gr^{r+i} A_{\crys}(-i)\otimes_{\bZ_p} H^i_{\et}(X_{\overline{K}},\bZ_p)_{\free})^{G_K}. 
\end{align*}
Existence of an inverse up to $t^i$ implies that the image of the latter homomorphism contains 
$(\gr^r A_{\crys}\otimes_{\bZ_p} H^i_{\et}(X_{\overline{K}},\bZ_p)_{\free})^{G_K}$. 

It also implies that any torsion element in $\gr^r M_{\dR}(H^i_{\et}(X_{\overline{K}},\bZ_p)_{\free})$ maps to zero under the natural homomorphism
\[
\gr^r M_{\dR}(H^i_{\et}(X_{\overline{K}},\bZ_p)_{\free})
\to
(\gr^i A_{\crys}(-i))^{G_K}\otimes_{O_K}\gr^r M_{\dR}(H^i_{\et}(X_{\overline{K}},\bZ_p)_{\free}).
\] 
This means that any such torsion element is killed by $p^{\frac{i}{p-1}}\cdot i !$. 
This gives a bound on the size of $\gr^r (M_{\dR})_{\torsion,v}$ in terms of $d$ and $h^{r,i-r}$. 
Therefore, $\#\gr^r (M_{\dR})_{\torsion}$ can be bounded by a constant only depending on $d$ and $h^{r,i-r}$. 

Then, we only need to consider lattices. 
Putting all homomorphisms together, we have
\[
H^{i-r}(X^\times,\Omega^r)\to
(\gr^i A_{\crys}(-i))^{G_K}\otimes_{O_K}
\gr^r M_{\dR}(H^i_{\et}(X_{\overline{K}},\bZ_p)_{\free}) 
\]
and
\[
\gr^r M_{\dR}(H^i_{\et}(X_{\overline{K}},\bZ_p)_{\free})\to
(\gr^d A_{\crys}(-d))^{G_K}\otimes_{O_K}H^{i-r}(X^\times,\Omega^r). 
\]
Note that we are using the identification $(\gr^0 A_{\crys})^{G_K}=O_K$. 

Then, it is clear that we can bound the difference between lattices, and such a bound only depends on $d$ and $h^{r,i-r}$.  
\end{proof}

\subsection{Comparison with the Faltings height}

Let $A_F$ be an abelian variety of dimension $g$ over $F$.  
We write $h_{\textnormal{Fal}}(A_F)$ for the \emph{stable} Faltings height of $A_F$. 
(This is called the \emph{geometric} height of $A_F$ in ~\cite{Faltings-Chai}*{p.169}.)
Note that if we use the normalization of metrics in Subsection $4.2$ to define the Faltings height,  then it equals $(\sqrt{2\pi})^g$ times the geometric height in \textit{loc.~cit.}. 

Let $M$ be a pure $\bZ$-motive attached to the first cohomology of $A_F$. 
We compare $h_{\textnormal{Fal}}(A_F)$ with the stable height $h(M)$ in our sense. 

\begin{prop}\label{AV}
The difference $\lvert h_{\textnormal{Fal}}(A_F)-h(M)\rvert$ is bounded by a constant only depending on $d=\dim A_F$. 
\end{prop}

\begin{rem}
One can ask if $h_{\textnormal{Fal}}(A_F)=h(M)$ holds. 
As explained below, it is reduced to comparison between Kisin modules and logarithmic cohomology. It is not known in general, especially if $A_F$ has semistable bad reduction at even finite places. 
\end{rem}

Recall that the Faltings height $h_\textnormal{Fal}(A_F)$ is the Arakelov degree of the arithmetic line bundle determined by the canonical bundle of the Neron model $A$ of $A_F$. 

Since we only consider stable heights, we may allow finite base extensions. 
If $A$ has potentially good reduction everywhere,  we obtain the proper smooth Neron model after a finite base extension. 
Then, we can apply Proposition ~\ref{Key}. 
However, $A$ is far from proper in general. 
We will settle this point. 

\begin{proof}[Proof of Proposition]
By the semistable reduction theorem, $A_F$ has semistable reduction everywhere after a finite base extension. 
Moreover, after a further finite base extension if necessary, one can construct a proper semistable model $X$ containing the Neron model $A$ of the base change of $A_F$ by means of Mumford's construction ~\cite{Kunnemann}. 

Since stable heights are stable under finite base extensions, 
we may assume $X$ exists over $F$. 
We write $X^\times$ for a log scheme with the log structure given by the special fiber.  
The computation of Faltings-Chai ~\cite{Faltings-Chai}*{VI. 2} can be applied to $X^\times$, and it shows that $H^{0}(X^\times,\Omega^r)$ is isomorphic to $H^0(A,\Omega^r)$. 

Now, the theorem can be proved as follows. 
One can show that the factor $\# L(M)_{\torsion}$ appeared in the definition of $h(M)$ equals one. 
Indeed, $\gr^0(M_{\dR,v})$ is torsion-free in any situation and, combined with duality,  we deduce that $\gr^1(M_{\dR}(T^{\vee})_v)$ is also torsion-free.  
Therefore, our height and Faltings' height are defined in the same way, and we only have to compare lattices. 

At almost all finite places, it follows from Theorem ~\ref{strong} that our lattice equals
$\det H^0(X^\times, \Omega^1)=H^0(A, \Omega^d)$. 
The latter lattice is used to define the Faltings height, so there is no difference.     
Even for other finite places, which are small relative to $d$, we can compare
$H^0(X^\times, \Omega^1)$ and $M_{\dR}$ as in the proof of Propostion ~\ref{Key}. 
Hence, we can bound the difference. 
\end{proof}

\section{Variations of Hodge Structure}
We review results on variations of Hodge structure. 
For simplicity, we restrict ourselves to algebro-geometric situations. 
We emphasize that a polarization plays an important role. 

In this section, we identify a smooth variety over $\bC$ and its associated complex manifold. 

\subsection{Positivity results}
Let $f\colon X_{\bC}\to S_\bC$ be a projective smooth morphism of complex smooth varieties.  
For any integer $i$, the relative $i$-th cohomology $R^{i}f_*\bZ$ is a local system on $S_\bC$ and it underlies a variation of integral Hodge structure of weight $i$. 
We write $\cV$ for its associated filtered vector bundle. 
Namely, its filtration is the Hodge filtration. 

Fix a relatively ample line bundle $L$ on $X_\bC$. 
We denote by $c_1(L)$ its first Chern class, which is a class in $H^2(X,\bZ)$. 
By the relative Hard Lefschetz theorem, we have the Lefschetz decomposition
\[
Rf^{i}_* \bQ=\bigoplus_{j\geq 0} c_1(L)^j P^{i-2j}. 
\]
Here, $P^{i-2j}$ is the primitive part of $Rf^{i-2j}_* \bQ$ with respect to $L$. 
This decomposition induces a decomposition of the variation of Hodge structure. 
In fact, each primitive part $P^*$ is a variation of polarized rational Hodge structure with a polarization $Q_L$ induced by $L$. 

We prefer to work with the full $i$-th cohomology. 
The polarization on $P^{i-2j}$ defines a polarization on $c_1(L)^j P^{i-2j}$. 
Therefore, we equip $Rf^{i}_* \bQ$ with a rational polarization via the Lefschetz decomposition. 
We can replace it by a scalar multiple so that it is integral. 

In ~\cite{Griffiths}*{(7.13)}, Griffiths considered the following line bundle:
\[
\cL_\bC :=\bigotimes_{r\in \bZ}(\det \gr^r(\cV))^{\otimes r}. 
\] 
He called it the ``canonical bundle" of the variation of Hodge structure. 
We would like to remark that this type of tensor product appears in the definition of heights of motives. 

\begin{prop}[Griffiths, Peters]\label{positivity}
Assume one of the following holds:
\begin{itemize}
\item $S_\bC$ is proper. 
\item $S_\bC$ is a Zariski open subsest of a projective smooth curve $\overline{S}_\bC$.   \end{itemize}  

Then, some power of $\cL_\bC$ is generated by global sections on $S_\bC$. 
Moreover, $\cL_\bC$ is ample on the image of the period map, see remarks below. 
\end{prop}

\begin{proof}
The corresponding statements for the primitive parts can be found in ~\cite{Griffiths}*{(9.7)} and ~\cite{Peters}*{(4.1)}. 
This in turn implies the desired statement. 
\end{proof}

We give supplemental, not completely precise, remarks. 

The variation of polarized (integral) Hodge structure gives the period map. 
Roughly speaking, it is a holomorphic map $S_\bC\to \Gamma\backslash D$. 
Here, $\Gamma\backslash D$ is a suitable Griffiths period domain. 

Over the period domain, there is the universal family of Hodge structure. 
(It may not satisfy the Griffiths transversality.) 
Hence, the line bundle $\cL_\bC$ on $\Gamma\backslash D$ makes sense. 
Moreover, $\cL_\bC$ over $S_\bC $ can be regarded as the pullback of $\cL_\bC$ over $\Gamma\backslash D$. 

The proposition above says that $\cL_\bC$ is ample on the image of the period map. 
In particular, the image of the period map is a quasi-projective variety. 
If the period map is finite, $\cL_\bC$ is ample on $S_\bC$. 

We say that a variation of Hodge structure on $S_\bC$ is nonisotrivial if it is not trivial on any finite covering of $S_\bC$. 
If $S_\bC$ is a (geometrically) connected open curve, then $\cL_\bC$ obtained from a nonisotrivial variation of Hodge structure is ample on $S_\bC$

\subsection{Singularities of metrics}
In the noncompact base case, we need to analyze asymptotic behaviors of metrics. 

Assume that $S_\bC$ is a Zariski open subsest of a projective smooth curve $\overline{S}_\bC$ and $X_\bC$ has log smooth degenerations over $\overline{S}_\bC$ with reduced fibers. 
This implies that the monodromy is unipotent, see ~\cite{Nakayama}*{(3.7)} for a more precise result on log smooth degenerations. 

Recall that the Hodge metric is defined by a positive definite hermitian form
\[
h_L(v,w)=Q_L(Cv,\overline{w}),
\]
where $C$ is the Weil operator, which acts by $(\sqrt{-1})^{i-2r}$ on $(i-r,r)$-components. 
It induces the Hodge metric on the determinant line bundle $\det(\Fil^r \cV)$ for any $r$. 

\begin{prop}[\cite{Peters}*{(2.2.1)}]
The Hodge metric on $\det(\Fil^r \cV)$ has logarithmic singularities. 
\end{prop}

See, e.g. ~\cite{Faltings-Chai}*{V.4.2} for the definition of logarithmic singularities.  
Peters' results is based on ~\cite{Zucker}. 

\begin{rem}\label{CKS}
Asymptotic analysis of the Hodge metric is a well-known application of $\SL_2$-orbit theorem ~\cite{Schmid}. 
In fact, this is also hidden in the proof of semi-ampleness of $\cL_\bC$ for the case $S_{\bC}$ is a curve. 

A generalization for an abstract variation of Hodge structure over a higher dimensional base can be found in  ~\cite{Cattani-Kaplan-Schmid}*{(5.21)}. 
\end{rem}

Since
$
\cL_\bC\cong \bigotimes_{r\geq 1} \det \Fil^r \cV, 
$
the Hodge metric $h_L$ induced on $\cL_\bC$ also has logarithmic singularities. 

There is a natural isomorphism
\[
\cL_\bC\otimes \overline{\cL}_\bC\cong \det(Rf^i_*\bQ)^{\otimes n} \otimes_\bZ \cO_{S_\bC}
\]
as in Subsection 4.2. 
We can define a hermitian metric $h$ on $\cL_\bC$ as follows: on each small open $U$, sections of the local system $\det(Rf^i_*\bZ)^{\otimes n}$ on $U$ is isomorphic to $\bZ$. 
The procedure in Subsection 4.2. gives a metric on $U$, and it glues together. 
Note that this metric is independent of polarizations. 

\begin{prop}\label{logarithmic singularity}
The hermitian metric $h$ has logarithmic singularities.  
\end{prop}

\begin{proof}
We claim that locally $h$ is a scalar multiple of $h_L$ on $\cL_\bC$. 
This follows from the fact that $h_L$ is a rational polarization.  
\end{proof}

\section{Heights in Families}\label{Heights in families}
Kato conjectured a behavior of his heights of motives parametrized by a curve, cf. ~\cite{Kato2}*{5.4}. 
We prove a version of it for our heights. 

\subsection{Statement}
Let $X_{F}$ and $S_F$ be quasi-projective smooth varieties over a number field $F$, and $f\colon X_F\to S_F$ a projective smooth morphism. 
For an integer $i$, we regard the $i$-th relative cohomology of $f$ as a family of $\bZ$-motives. 
We denote it by $M$. 
For each closed point $t\in S_F$, it defines a pure $\bZ$-motive $M_t$ over its residue field $F(t)$. 
Note that we ignore the torsion part of the Betti cohomology.  
We will only consider this form of families of $\bZ$-motives.  

We put the following assumptions on degenerations: 
\begin{quote}
($\ast$): there exist a proper flat scheme $\overline{S}$~(resp.~$\overline{X}$) over $O_F$ containing $S_F$ (resp. $X_F$), a fine saturated log structure on $\overline{S}$~(resp.~$\overline{X}$) which is trivial on $S_F$~(resp.~$X_F$), a proper flat (in the non-logarithmic sense) morphism $\overline{f}\colon\overline{X}\to\overline{S}$ of log schemes whose restriction to $S_F$ is $f$. 

We require that $\overline{f}$ is log smooth and its relative log Hodge cohomology is locally free. 
\end{quote}

We are interested in heights of motives $M_t$. 
We write $h_M (t)$ for the function $t\mapsto h(M_t)$. 
On the other hand, for a line bundle $\cL$ on $S_F$, we have the height function $h_{\cL}(t)$ attached to $\cL$. (well-defined up to bounded functions on the set of closed points $\lvert S_F\rvert$.)

For two functions $f_1$ and $f_2$ on $\lvert S_F\rvert$,  we write $f_1\sim f_2$ if  their difference is a bounded function on $\lvert S_F\rvert$. 

When $S_F$ is a curve, Kato conjectured there exists a line bundle $\cL$ such that $h_M(t)\sim h_\cL (t)$.  
Technically speaking, this should be the case if $S_F=\overline{S}_F$. 
In general, the height function should be attached to an arithmetic line bundle, namely, a line bundle over an integral model equipped with metrics. 

He essentially predicted that $\cL$ is a variation of $L(M_t)$. 
Note that $M_\bQ$ has the de Rham realization $M_{\dR}$, which is a vector bundle on $S$ with the Hodge filtration. 
Then, a line bundle 
\[
\cL :=L(M)_{\bQ} :=\bigotimes_{r\in \bZ}(\det \gr^r(M_{\dR}))^{\otimes r}. 
\] 
is the one we are seeking for. 
(Kato called its degree the geometric height of $M$.)
For each embedding $F\to \bC$, we equip $\cL$ with a metric as in the previous section. 
This defines the height function $h_\cL (t)$. 

We say that a function $f$ on $\lvert S_F\rvert$ satisfies the finiteness property if, for any real numbers $c_1$ and $c_2$, there are finitely many closed points $t$ such that $f(t) <c_1$ and $[F(t)\colon F]< c_2$. 

We prove the following:

\begin{thm}\label{family}
If $(\ast)$ is satisfied, then $h_M(t)\sim h_\cL (t)$. 
If we further assume one of the following
\begin{itemize}
\item $S_F=\overline{S}_F$, or
\item $\overline{S}_F$ is a projective smooth curve and $\overline{f}_F$ has geometrically reduced fibers. 
\end{itemize}

Then $h_M(t)$ is bounded below. 
Moreover, $h_M (t)$ satisfies the finiteness property in the above sense if the period mapping associated to a variation of (polarized) Hodge structure which realize $M$ is a finite morphism. 
\end{thm}

\subsection{Proof}
Take $\overline{X}$ and $\overline{S}$ as in ($\ast$). 
Its relative log Hodge cohomology commutes with base change because it is assumed to be locally free and sheaves of logarithmic differential forms commutes with fiber products in the category of fine saturated log schemes. 

Given a closed point $t$ of $S_F$, We claim that there exists a extension $F'$ of $F(t)$ such that $(\overline{X}_{O_{F(t)}})_{O_{F'}}$ becomes semistable everywhere after log blow-ups. 
This follows from Theorem~\ref{potentially semistable} and inductive approximations. 
Namely, we use the following consequence of Kranser's lemma: 
given a number field $F$, a finite place $v$ and a finite extension $K'$ of $F_v$, there exist a finite extension $F'$ of $F$ and a unique finite place $v'$ above $v$ such that $F'_{v'}$ is isomorphic to $K'$. 
 
Then, it is a consequence of the proof of Corollary~\ref{log smooth}, and Proposition~\ref{Key} that $h_M(t)\sim h_\cL(t)$. 

It is a standard argument that Proposition ~\ref{positivity}, remarks following it, and Proposition ~\ref{logarithmic singularity} implies the rest of statement.  
This argument is essentially contained in ~\cite{Faltings83}. 
However, for the convenience of the reader, we will give more details. 

From now on, we will write $\overline{S}_1$ instead of $\overline{S}$.  
First note that $\cL$ extends to a line bundle $\cL_1$ on $\overline{S}_1$ by means of the log Hodge cohomology. 
The line bundle $\cL_1$ (or its power) may not be generated by global sections. 

On the other hand, take $n$ such that $\cL^{n}$ is generated by global sections.  
Then, we have another proper flat model $\overline{S}_2$ of $\overline{S}_F$ such that $\cL^n$ extends to a line bundle $\cL_2$ on $\overline{S}_2$ generated by global sections. 

There exists another proper flat model $\overline{S}_3$ which dominates both $\overline{S}_1$ and $\overline{S}_2$. 
We pull back everything to $\overline{S}_3$ and will drop subscript 3. 
Since $\cL_1$ and $\cL_2$ are models of the same line bundle, there exists a positive integer $N$ and homomorphisms $\cL_1^{N}\to\cL_2$, $\cL_2\to\cL_1^{N}$ such that two compositions are multiplication by some integer, and induces isomorphisms over $F$.   
Hence, height functions attached to $\cL_1$ and $\cL_2$, using the same metrics, only differ by a bounded function. 

Finally, we check that the height function of $\cL_2$ as an arithmetic line bundle on $\overline{S}_2$ satisfies the desired property. 
If $S_F=\overline{S}_F$, this is well-known. 
(In fact, any metric works.)

Suppose $S_F$ is a nonproper smooth curve. 
If $\cL$ has the degree $0$ on a connected component, then the height function $h_\cL$ is bounded on that component. 
Otherwise, the degree is positive, hence ample. 
Every component has the positive degree if the period mappings are finite. 
Since the metric has logarithmic singularities by Proposition ~\ref{logarithmic singularity}, the finiteness property follows from ~\cite{Faltings83}*{Lemma 3}. 

\begin{rem}\label{Siegel}
If $M$ is attached to the first cohomology of an abelian scheme, the theorem is a weaker version of Faltings' theorem via Proposition ~\ref{AV}. 
In fact, we can repeat his argument in our setting as follows. 

The first part of the statement can be applied to suitable arithmetic toroidal compactifications of the moduli space of principal abelian varieties and the universal family, if we ignore a small problem that the moduli space is only a Deligne-Mumford stack. 

The metric on $\cL$ has logarithmic singularities ~\cite{Faltings-Chai}*{V. 4.5}. 
(This can be generalized, cf.~Remark~\ref{CKS}.)
More crucially, a power of $\cL$ comes from an ample line bundle on the coarse moduli space~\cite{Faltings-Chai}*{V. 2.3}. 
These imply the finiteness property for $h_M(t)$. 
\end{rem}

\section{Torsion $p$-adic Hodge Theory}

This technical section discusses torsion Galois representations. 
The main purpose is to supply a few results not available in literature which will replace the theory of finite flat group schemes used in Raynaud's paper ~\cite{Raynaud}. 

\subsection{Torsion Kisin modules}

Recall the definition of a torsion Kisin module:

\begin{defn}
A torsion Kisin module $\fM$ of height $\leq w$ is a quotient of two Kisin modules of height $\leq w$ which is killed by some $p$-power, namely $\fM=\fN/\fN'$ with $T_\fS(\fN')\otimes_{\bZ_p}\bQ_p \overset{\cong}{\to} T_\fS(\fN)\otimes_{\bZ_p}\bQ_p$. 
\end{defn}

\begin{prop}[cf.~\cite{Liu:torsion}*{Corollary 2.3.4}]
There is an exact tensor contravariant functor
\[
\fM\mapsto T_{\fS}(\fM)=\Hom_\varphi(\fM, \bQ_p/\bZ_p \otimes_{\bZ_p}W(R))
\] from the category of torsion Kisin modules to the category of torsion $\bZ_p$-representations of $G_\infty$. 
\end{prop}

\begin{rem}
As the notation suggests, the above $T_\fS$ is compatible with the $T_\fS$ for free Kisin modules, see the proposition below. 
The contravariant functor $T_{\fS}$ is not fully faithful for torsion Kisin modules in general. 
\end{rem}

\begin{prop}[\cite{Liu:torsion}*{Section 2.}]\label{exact sequence 0}
Let $V$ be a semstable representation of $G_K$ and let $T\subset T'$ be lattices in $V$ which are $G_\infty$-stable. 
We write $\fN$ and $\fN'$ for Kisin modules corresponding to $T$ and $T'$ respectively. 
Then, $\fN/\fN'$ is a torsion Kisin module and there is an exact sequence
\[
0\to T \to T' \to T_{\fS}(\fN/\fN') \to 0. 
\]
\end{prop}

Let $\fM$ be a torsion or free Kisin module of height $\leq w$. 
Note that $1\otimes\varphi\colon \varphi^*\fM=\fS\otimes_{\varphi, \fS}\fM\to\fM$ is injective, cf.~\cite{Liu:torsion}*{Proposition 2.3.2}. 
We define a filtration on $\varphi^*\fM$ by
\[
\Fil^r\varphi^*\fM=\{x\in \varphi^*\fM\mid (1\otimes\varphi)(x) \in E(u)^r \fM\}. 
\]

Recall the homomorphism $\fS\to O_K;u\mapsto \pi$. 
We write $\fM_{\dR}$ for $O_K\otimes_{\fS}\varphi^*\fM$, and equip a filtration on $\fM_{\dR}$ by the image of the filtration on $\varphi^*\fM$. 
If $\fM$ is a free Kisin module coming from a lattice $T$ in a semistable representation, $\fM_{\dR}=M_{\dR}(T)$. 

Suppose we are in the situation of Proposition ~\ref{exact sequence 0}. 
Since any torsion Kisin module is $E(u)$-torsion free ~\cite{Liu:torsion}*{Proposition 2.3.2}, 
for any $r$, there is a commutative diagram of exact sequences
\[
\begin{split}
 \xymatrix{
0\ar[r] & E(u)^r\fN' \ar[r] & E(u)^r\fN\ar[r] \ar[r] & E(u)^r(\fN/\fN') \ar[r] & 0 \\
0\ar[r] & \Fil^r \varphi^*\fN' \ar[r] \ar[u] & \Fil^r \varphi^*\fN \ar[r] \ar[u] & 
\Fil^r \varphi^*(\fN/\fN')\ar[r] \ar[u] & 0. \\
} 
\end{split}
\]

Thus, 
we obtain an exact sequence
\[
0\to \gr^r \varphi^*\fN' \to\gr^r \varphi^*\fN \to \gr^r \varphi^*(\fN/\fN') \to 0. 
\]

\begin{lem}\label{refined filtration}
Let $\fM$ be a free or torsion Kisin module. 
Then, there is a functorial construction of a decreasing filtration $F$ on $\gr^r \varphi^*\fM$ such that
$F^0(\gr^r\varphi^*\fM)=\gr^r \varphi^*\fM$, 
$F^{r+1}(\gr^r\varphi^*\fM)=0$ and, for any $j$ between $0$ and $r$, 
\[
F^j(\gr^r \varphi^*\fM)/F^{j+1}(\gr^r \varphi^*\fM) \cong \gr^{r-j}(\fM_{\dR}). 
\]
\end{lem}

\begin{proof}
We define a filtration $F$ on $\Fil^r\varphi^*\fM$ by
\[
F^j(\Fil^r\varphi^*\fM)=(\varphi^*\fM\cap E(u)^{r+1}\fM)+(E(u)^j\varphi^*\fM\cap E(u)^r\fM).
\]
So, $F^0(\Fil^r\varphi^*\fM)=\Fil^r\varphi^*\fM$ and $F^{r+1}(\Fil^r\varphi^*\fM)=\Fil^{r+1}\varphi^*\fM$. 
Then, we observe that
\begin{align*}
&\frac{F^j(\Fil^r\varphi^*\fM)}{F^{j+1}(\Fil^r\varphi^*\fM)}
=\frac{(\varphi^*\fM\cap E(u)^{r+1}\fM)+(E(u)^j\varphi^*\fM\cap E(u)^r\fM)}{(\varphi^*\fM\cap E(u)^{r+1}\fM)+(E(u)^{j+1}\varphi^*\fM\cap E(u)^r\fM)} \\
&\cong
\frac{E(u)^j\varphi^*\fM\cap E(u)^r\fM}{(E(u)^j\varphi^*\fM\cap E(u)^r\fM)\cap ((\varphi^*\fM\cap E(u)^{r+1}\fM)+ (E(u)^{j+1}\varphi^*\fM\cap E(u)^r\fM))} \\
&=
\frac{E(u)^j\varphi^*\fM\cap E(u)^r\fM}{(E(u)^j\varphi^*\fM\cap E(u)^{r+1}\fM)+(E(u)^{j+1}\varphi^*\fM\cap E(u)^r\fM)} \\
&\cong
\frac{\varphi^*\fM\cap E(u)^{r-j}\fM}{(\varphi^*\fM\cap E(u)^{r-j+1}\fM)+(E(u)\varphi^*\fM\cap E(u)^{r-j}\fM)}
\cong \gr^{r-j}(\fM_{\dR}). 
\end{align*}
The filtration $F$ induces a desired filtration on $\gr^r \varphi^*\fM$.  
\end{proof}

By induction on $r$, we can deduce the following proposition. 

\begin{prop}\label{exact sequence}
In the situation of Proposition~\ref{exact sequence 0}, 
a sequence 
\[
0\to \gr^r(M_{\dR}(T'))\to \gr^r(M_{\dR}(T))\to \gr^r((\fN/\fN')_{\dR})\to 0
\]
is exact. 
\end{prop}

\subsection{Galois actions on determinants}
If the associated Galois representations is free of level $m$, we can compute the inertia action (or its restriction to $G_\infty$) on its determinant under certain technical assumptions. 
This will be an important input for our application to heights of motives.  

Let $\fM$ be a torsion Kisin module of height $\leq w$ which is a free $\fS_m$-module of rank $h$.  
Then, $T_{\fS}(\fM)$ is a free $\bZ/p^m\bZ$-representation of $G_{\infty}$. 

It is easy to see that $\fM_{\dR}$ is a free $W_m$-module equipped with the filtration by free $W_m$-modules. 
Moreover, its graded pieces are also free $W_m$-modules. 

Set $d_r=\rk_{W_m}\gr^r(\fM_{\dR})$ and $d=\sum_{r=1}^{w} r\cdot d_r$. 
We assume $w\neq 0$. 
(Otherwise, Galois actions are trivial and $d=0$.)

\begin{thm}\label{determinant}
Suppose $k$ is algebraically closed and $m\geq 2^{hw-1}m'$ for some integer $m'\geq 2$.   

Then, $\frac{d}{e}$ is an interger and
$\det (T_{\fS}(\fM))$ modulo $p^{m'}$ is isomorphic to $(\bZ/p^{m'}\bZ)(\frac{d}{e})$ as representations of $G_{\infty}$. 
\end{thm}

\begin{rem}
Since $T_{\fS}$ commutes with extensions of $k$, one can use this theorem to compute actions of $G_{\infty}\cap I_K$ if $k$ is not algebraically closed. 
\end{rem}

\begin{lem}\label{exponent}
An equality $d=\rk_{W_m}\fM/\varphi^*\fM$ holds. 
In particular, an iequality $d\leq ehw$ holds. 
\end{lem}

\begin{proof}
We have inclusions $E(u)^w\fM\subset \varphi^*\fM\subset \fM$ and an equality $\rk_{W_m}\fM/E(u)^w\fM=ehw$. 
By Lemma ~\ref{refined filtration}, one can see that 
\[
\rk_{W_m}\Fil^r\varphi^*\fM/\Fil^{r+1}\varphi^*\fM=\sum_{j=0}^r d_j.
\] 
(Note that $\Fil^r\varphi^*\fM/\Fil^{r+1}\varphi^*\fM$ is also a free $W_m$-module.)

Since $\Fil^w\varphi^*\fM=E(u)^w\fM$, 
\[
\rk_{W_m}\varphi^*\fM/E(u)^w\fM
=\sum_{r=0}^{w-1}\sum_{j=0}^r d_j=\sum_{r=0}^w (w-r)d_r. 
\]
As $ed=\sum_{r=0}^w d_r$, we obtain $d=\rk_{W_m}\fM/\varphi^*\fM$.  
\end{proof}

\begin{proof}[Proof of Theorem]

As $T_\fS$ is an exact tensor contravariant functor, it commutes with exterior powers. 
Since 
\[
\rk_{W_m}\fM/\varphi^*\fM=\rk_{W_m}\det(\fM)/\varphi^*\det(\fM), 
\]
we may assume $\fM$ has rank $1$ and height $\leq hw$. 

Take a basis $x$ of $\fM$ and define $f\in\fS_{m}$ by $\varphi(x)=fx$. 
Then, there exists $g\in \fS_{m}$ such that $fg=E(u)^{hw}$. 

We need the following form of Eisenstein criterion. 
(This is essentially appeared in the proof of ~\cite{Liu:torsion}*{Lemma 4.2.3}.)

\begin{lem}
Let $m$ and $m'$ be positive integers such that $m\geq 2^{w-1}m'$ and $m'\geq 2$.   

Suppose $f$ and $g$ are polynomials with coefficients in $\fS_{m}$, and $fg=E(u)^w$. 
Then, modulo $p^{m'}$, $f$ and $g$ are powers of $E(u)$ up to units. 
\end{lem}

\begin{proof}
We prove this by induction on $w$. 

Suppose $w=1$. 
We will prove $f$ or $g$ is a unit.  
The reduction modulo $p$ of $f$ and $g$ can be written as $u^a$ and $u^b$ up to units respectively. 
Hence, up to units, $fg$ equals both $E(u)$ and $(u^a+pf_0)(u^b+pg_0)$ with some $f_0$ and $g_0$. 
If $a,b>0$, one gets contradiction by looking at constant terms modulo $p^2$. 
So $ab=0$, and $f$ or $g$ is a unit. 

Let $w>1$. 
As $E(u)$ is a nonzerodivisor, it suffices to prove that $E(u)$ divides $f$ or $g$ modulo $p^{[\frac{m'}{2}]}$ to make induction work. 
As $E(u)$ is monic, write $f=E(u)f_1+f_0$ and $g=E(u)g_1+g_0$ by polynomial division. 
Then, $E(u)$ divides $f_0g_0$ in $\fS_{m}$. 
We can use ~\cite{Liu:torsion}*{Lemma 4.2.2} to conclude that $E(u)$ divides $f_0$ or $g_0$ modulo $p^{[\frac{m'}{2}]}$, and we are done. 
\end{proof}

By the previous lemma, $f=E(u)^{d'}f'$ for some $d'$ and some unit $f'$. 
So, $\rk_{W_{m'}}\fM/\varphi^*\fM=d'e$ and $d=d'e$.  

Choosing a lift $\tilde{f'}$ of $f'$,  one can define a free Kisin module $\tilde{\fM}$ of rank $1$ by $\varphi(\tilde{x})=E(u)^r\tilde{f'}\tilde{x}$. 
Note that $\tilde{\fM}/p^{m'}\tilde{\fM}$ is isomorphic to $\fM /p^{m'}\fM$. 

It is known that any Kisin module of rank $1$ comes from a weakly admissible filtered $\varphi$-module. 
Therefore, $T_{\fS}(\tilde{\fM})$ is a power of the cyclotomic character and, in fact, it is isomorphic to $\bZ_p(d')$. 
Hence, $T_{\fS}(\fM)/p^{m'}T_{\fS}(\fM)$ is isomorphic to $(\bZ/p^{m'}\bZ )(d')$. 
\end{proof}

\begin{rem}\label{tame:determinant}
The proof above can be used to show the following analogous statement if $m=1$. 
Let $\theta_1$ be the tame fundamental character of level 1, namely
$\sigma (\pi^{\frac{1}{p-1}})=\theta_1 (\sigma) \pi^{\frac{1}{p-1}}$ for $\sigma \in G_\infty$ and any choice of a $(p-1)$-th root of $\pi$. 
The character $\theta_1$ is valued in $\mu_{p-1}$, which is canonically isomorphic to $\bF^{\times}_p\subset k$. 
It acts on $\bF_p=\bZ/p\bZ$ naturally. 

Then, $\det (T_{\fS}(\fM))\cong \theta_{1}^{d}$ as representations of $G_{\infty}$. 
Indeed, as in the proof above, it is reduced to the rank $1$ case. 
Using the same notation in the proof, it is easy to see $f=u^{d}$ up to a unit, and one can calculate the associated $\bZ/p\bZ$-representation. 

We remark that two calculations are compatible because $\theta_{1}^e$ is exactly $\bZ/p\bZ(1)$. 
\end{rem}

\begin{rem}\label{truncated BT}
If $w=1$ and $m\geq 2$, a Kisin module which is a free $\fS_m$-module corresponds to a truncated $p$-divisible (or Barsotti-Tate) group of level $m$ over $O_K$. 

Therefore, a result of Illusie ~\cite{Illusie}*{4.10} implies that $\frac{d}{e}$ is an integer and
$\det (T_{\fS}(\fM))$ is isomorphic to $(\bZ/p^m\bZ)(\frac{d}{e})$. 
(Illusie's result is used in Raynaud's paper ~\cite{Raynaud}.) 

Our proof of the theorem above does not recover this result, since our calculation depends on $h$. 
It is natural to wonder how to improve it when $w\geq 2$. 
\end{rem}

\subsection{Fontaine-Laffaille modules}

Under the Fontaine-Laffaille condition, we can prove a stronger result. 
It is important to keep track of the whole action of $G_K$.  

\begin{defn}
A torsion crystalline representation of $G_K$ is a torsion $\bZ_p$-representation of the form $T'/T$ with $G_K$-stable lattices $T\subset T'$ in a crystalline representation of $G_K$.  
\end{defn}

Even when we can classify $G_K$-stable lattices,  it is more difficult to study torsion crystalline representations. 
If $e=1$ and the weight is small, we can use Fontaine-Laffaille modules. 

We assume $e=1$ and $w<p-1$. 

\begin{defn}
A Fontaine-Laffaille module is the following data:
\begin{itemize}
\item a filtered $W$-module $M$ with $\Fil^0=M$ and $\Fil^{w+1}M=0$.
\item Semilinear homomorphisms $\varphi_r\colon\Fil^r M\to M$ such that
$p\varphi_{r+1}=\varphi_r$ on $\Fil^{r+1}M$ and $\sum_r \varphi_r(\Fil^r M)=M$.  
\end{itemize}
\end{defn}

On $\Fil^r A_{\crys}$ with $r<p-1$, $\varphi$ is divisible by $p^r$ and $\frac{\varphi}{p^r}$ is well-defined. 
For a Fontaine-Laffaille module $M$ of finite length, 
we associate a Galois representation 
\[
T_W(M)=\Hom_{\Fil,\varphi}(M, \bQ_p/\bZ_p\otimes_{\bZ_p}A_{\crys}).
\]
Here, it is understood that elements of $T_W(M)$ commute with $\varphi_r$. 
This is known to be a fully faithful exact contravariant functor. 

\begin{exam}
Let $T$ be a $G_K$-stable lattice in a crystalline representation of $G_K$ with the Hodge-Tate weights in $\{0,\dots,w\}$. 
If $M$ is the corresponding strongly divisible lattice, one can give the structure of a Fontaine-Laffaille module on $M$ via the formula $\varphi^r=\frac{\varphi}{p^r}$, which is well-defined on $\Fil^r M$. 

Given a $G_K$-stable lattice $T'$ containing $T$, we denote by $M'$ the corresponding strongly divisible lattice. 
The quotient $M/M'$ has the structure of a Fontaine-Laffaille module, and
$T_W(M/M')\cong T'/T$. 
\end{exam}

Let $M$ be a Fontaine-Laffaille module which is a free $W_m$-module of rank $h$. 
Set $d_W=\sum_r r\cdot \rk_{W_m}\gr^r(M/M')$. 

\begin{prop}
Suppose $hw<p-1$ and $k$ is algebraically closed. 
Then, $\det (T_W(M))\cong \bZ/p^m\bZ(d_W)$. 
\end{prop}

\begin{proof}
By the assumption, $\det (M)$ has the structure of a Fontaine-Laffaille module. 
There is a natural homomorphism $\det (T_W(M))\to T_W(\det (M))$ as representations of $G_K$. 
We first show this is an isomorphism. 
It is known that both sets have the cardinality $p^m$. 
To show it is an isomorphism, it suffices to check modulo $p$.  
Since $T_W$ and determinants commute with the reduction modulo $p$, 
we may assume $m=1$. 
Then, one can verify it by using the classification of simple Fontaine-Laffaille modules of finite length  ~\cite{Fontaine-Laffaille}. 

There is a unique $r$ such that $\Fil^r \det(M)\neq 0$. In fact, $r=d_W$.   
It follows that $T_W(\det(M))\cong\bZ/p^m\bZ(d_W)$. 
\end{proof}

We interpret the above proposition in terms of torsion Kisin modules. 
Take two $G_K$-stable lattices $T\subset T'$ as in the example above. 
From them, choosing a uniformizer and its $p$-power roots, we get an exact sequences of Kisin modules
\[
0\to \fM' \to \fM \to \fM/\fM'\to 0. 
\]

Recall that $d=\sum_{r=1}^{w} r\cdot \rk_{W_m}\gr^r((\fM/\fM')_{\dR})$. 

\begin{prop}\label{Fontaine-Laffaille:determinant}
Assume that $hw<p-1$ and $T'/T$ is a free $\bZ/p^m\bZ$-module of rank $h$. 
Then, $d_W=d$. 
In particular, $\det (T'/T)$ is isomorphic to $\bZ/p^m\bZ(d)$ if $k$ is algebraically closed. 
\end{prop}

\begin{proof}
We have two strongly divisible lattices $M'\subset M\subset D$ corresponding to $T\subset T'$. 
An exact sequence
\[
0\to M' \to M \to M/M' \to 0
\]
is naturally identified to
\[
0 \to \fM'_{\dR} \to \fM_{\dR} \to (\fM/\fM')_{\dR} \to 0. 
\]
Therefore, $d=d_W$. 

The previous proposition implies the last statement because $T_W(M/M')$ is isomorphic to $T'/T$.
\end{proof}

\begin{exam}
Set $\fM=\fS_2\cdot a$ with $\varphi(a)=u\cdot a$, which is a torsion Kisin module of height $\leq 2$. 
The Galois representation $T_\fS(\fM)$ is a free $\bZ/p^2\bZ$-module of rank $1$, but it may not be isomorphic to a power of the cyclotomic character. 
If $e=1$ and $p\geq 5$, Fontaine-Laffaille theory excludes it when one wants to extend the $G_{\infty}$-action on $T_\fS(\fM)$ to a $G_K$-action so that it is torsion crystalline.
\end{exam}

\subsection{Rigidity and its consequence}
Raynaud ~\cite{Raynaud}*{3.3.1} proved the following rigidity result for truncated $p$-divisible groups: there exists a constant $c$ such that any morphism between truncated $p$-divisible groups of level $m\geq c$ over $O_K$ which induces an isomorphism between generic fibers is an isomorphism. 
(His constant $c$ depends on the discriminant of $O_K$.)

We prove an analogous rigidity statement~(Proposition~\ref{rigidity}) for torsion Kisin modules of higher weights. 
It seems that Raynaud's idea of the proof in \textit{loc.~.cit} does not work in this context.  
Instead we generalize Liu's arugement in ~\cite{Liu:BT}*{Proposition 5.2.2}. 
He found a different constant depending on the ramification index of $O_K$ and the heights of truncated $p$-divisible groups. 
There are also many other results of this direction in literature, but this seems to be the most suitable one for our purpose. 

For application to the Faltings height, Raynaud actually did not use rigidity itself but used its consequence~\cite{Raynaud}*{3.4.5}. 
Namely, he deduced from the rigidity that there exists a constant $c'$ such that for any finite flat group scheme $G$ over $O_K$ whose generic fiber is a truncated $p$-divisible group of level $m\geq 2c'+1$, $G[p^{m-c'}]/G[p^{c'}]$ is a truncated $p$-divisible group of level $m-2c'$ over $O_K$. 
We also discuss a generalization of this claim~(Corollary~\ref{free}). 

\begin{rem}
Results of the form of Corollary~\ref{free} has been known. 
For instance, such a statement can be found in the proof of ~\cite{Liu:torsion}*{Lemma 4.3.1}. 

Our argument, though it is closely related to \textit{loc.~cit.}, gives a better bound in general if the rank of a torsion Kisin module is bounded. 
Since this assumption holds in our application later, we think it is reasonable to include the argument in this paper. 
\end{rem}

Now, we state the rigidity result. 

\begin{prop}\label{rigidity}
Let $f\colon \fM\to\fM'$ be a homomorphism of torsion Kisin modules of height $\leq w$ which are free $\fS_m$-modules of rank $h$. 
Let $\lambda$ be the biggest integer such that $K$ contains primitive $p^{\lambda}$-th roots of unity and $c$ the smallest integer greater than or equal to
$(\log_p (hw)+\lambda)$.

If $f$ induces an isomorphism $T_{\fS}(\fM)\overset{\cong}{\to} T_{\fS}(\fM')$ and $m\geq 2^{hw-1}c$, then $f$ is an isomorphism. 
\end{prop}

\begin{rem}
An equality $\lambda\leq v_p (e)+1$ holds.
Here,  $v_p$ is a $p$-adic valuation with $v_p(p)=1$.  
\end{rem}

\begin{proof}
We may assume $k$ is algebraically closed. 
By Theorem ~\ref{determinant}, Galois actions on determinants of $T_{\fS}(\fM)$ and $T_{\fS}(\fM')$ modulo $p^c$ are powers of the cyclotomic character with exponents $\frac{d}{e}$ and $\frac{d'}{e}$ respectively. 
Here, 
\[
d=\rk_{W_m}\fM/\varphi^*\fM \quad \textnormal{and} \quad d'=\rk_{W_m}\fM'/\varphi^*\fM'. 
\]
(Note that Lemma ~\ref{exponent} is used.)

Then, it follows that $\frac{d}{e}\equiv \frac{d'}{e} (\textnormal{mod}\; p^{c-\lambda})$. 
Since $d, d'\leq ehw$, we conclude that $d=d'$.
By the assumption, $\fM'/\fM$ is killed by $u$ and  $\fM'/\fM$ is a finite set. 
So, it must be zero.   
\end{proof}

\begin{rem}
In the case of finite flat group schemes, Raynaud proved a similar statement and his constant is, writing $\textnormal{diff}_{K/K_0}$ for the different ideal of $K$ over $K_0$, the integer part of $(\frac{1}{p-1}+v_p(\textnormal{diff}_{K/K_0}))$. 

In the context of torsion Kisin modules, it corresponds to the case $w=1$. 
If $w=1$, our constant can be improved to $c$ by using Remark ~\ref{truncated BT} instead of Theorem ~\ref{determinant} in the proof above. (This is the bound Liu obtained in ~\cite{Liu:BT}*{Corollary 5.2.3}.)

These two constants may be compared by a well-known inequality
\[
v_p(\textnormal{diff}_{K/K_0})\leq 1-\frac{1}{e}+v_p (e).
\] 
\end{rem}

\begin{rem}
As in the proposition, rigidity is related to how many $p$-power roots of unity are contained in $K$, but the example below indicates a different nature. 
\end{rem}

\begin{exam}
Let $\fM$ be a $\fS$-submodule of $\fS_2$ generated by $u$, which is a free $\fS_2$-module of rank $1$. 
If $e=1$, $\fM$ is a torison Kisin module of height $\leq p$, and the inclusion $\fM\subset \fS_2$ induces an isomorphism between Galois representations. 
So, rigidity does not hold in this case.  
\end{exam}

Let $\fM$ be a torsion Kisin module such that $T_\fS(\fM)$ is a free $\bZ/p^m\bZ$-module. 
We denote by $c'$ the greatest integer less than or equal to $\frac{ehw}{p-1}$, and set $\delta=2^{hw-1}cc'$. 

For nonnegative integers $a<b$, we define $\fM^{a,b}$ as the kernel of multiplication by $p^{b-a}$ from $p^a\fM$ to $p^b\fM$. 

\begin{cor}\label{free}
If $m>2\delta$, then $\fM^{\delta,m-\delta}$ is a free $\fS_{m-2\delta}$-module. 
\end{cor}

\begin{proof}
We have an increasing sequence of torsion Kisin modules
\[
\fM^{m-1,m}\subseteq \fM^{m-2,m-1}\subseteq\cdots\subseteq \fM^{0,1}. 
\]
All these inclusions induce isomorphisms of associated Galois representations. 
Then, it is easy to see that there are at most $c'+1$ distinct terms, see the proof of ~\cite{Caruso-Liu}*{Lemma 3.2.4}. 
Therefore, the above sequence can be divided into at most $c'+1$ subsequences so that each subsequence consists of equalities. 

Suppose there are two subsequences of length  $\geq 2^{hw-1}c$, and 
they have the right end terms $\fM^{i_1,i_1+1}$ and $\fM^{i_2,i_2+1}$ with $i_1<i_2$. 
For $j=1$ or $2$, we have a torsion Kisin module $\fM^{i_j, i_j+2^{hw-1}c}$ of of height $\leq w$. By ~\cite{Liu:torsion}*{4.2.4}, it is a free $\fS_{2^{hw-1}c}$-module.

Consider multiplication by $p^{i_2-i_1}$ from $\fM^{i_1,i_1+2^{hw-1}c}$ to $\fM^{i_2,i_2+2^{hw-1}c}$. 
It is not an isomorphism, but it induces an isomorphism between $T_\fS(\fM^{i_1,i_1+2^{hw-1}c})$ and $T_\fS(\fM^{i_2,i_2+2^{hw-1}c})$. 
This contradicts to the rigidity, Proposition ~\ref{rigidity}. 

So, there is a unique maximal subsequence of length $\geq 2^{hw-1}c$ which consists of equalities. 
Combined with ~\cite{Liu:torsion}*{Lemma 4.2.4}, this implies the desired statement. 
\end{proof}

\section{Good Isogenies}
One of the key ingredients in Faltings' paper ~\cite{Faltings83} is to control a behavior of heights under isogenies. 
In particular, he proved that certain isogenies preserve heights, and later this is refined by Raynaud ~\cite{Raynaud}.  
Inspired by Raynaud, we introduce a notion of good isogenies (Definition ~\ref{good isogeney})  and prove that good isogenies preserve heights under certain conditions. 
Note that all these arguments are application of the purity, namely the Weil conjecture ~\cite{Deligne}. 

\subsection{Heights and isogenies}

Let $M=(M_\bQ, T)$ and $M'=(M_\bQ, T')$ be pure $\bZ$-motives of weight $w$ over a number field $F$. 
An inclusion $f\colon T'\to T$ is called an isogeny. 
In the rest of this section, we assume $M_\bQ$ has semistable reduction everywhere, and we study the difference $h(M)-h(M')$. 

We also assume that all the Hodge-Tate weights of the \'etale realizations of $M$ are nonpositive. 

For each finite place $v$ of $F$ above a prime number $p$, choose a unifomizer of $F_v$ and its $p$-power roots.   
One can associate an exact sequence of Kisin modules of height $\leq w$
\[
0\to \fM'_v\to \fM_v\to \fM_v/\fM'_v\to 0, 
\]
which induces an exact sequence of associated Galois representations
\[
0\to T^{\vee}_p\to T'^{\vee}_p\to T'^{\vee}_p/T^{\vee}_p\to 0. 
\]

We use the following notation:
\[
\Deg_{\et,p}=\# (T_p/T'_p), \; \Deg_{\dR,v}=\# ((\fM_v/\fM'_v)/\varphi^*(\fM_v/\fM'_v)). 
\]

\begin{prop}\label{formula}
The following formula holds
\[
h(M)-h(M')=\frac{1}{[F\colon\bQ]}\sum_{v}\log (\Deg_{\dR,v})-\frac{w}{2}\sum_p \log (\Deg_{\et,p}). 
\]
\end{prop}

\begin{proof}
Take an element $s\in L(M')$. 
By the definition of the height, we only need to calculate
$\frac{\# L(M)_{\torsion}}{\# L(M')_{\torsion}}\cdot\#(L(M)/L(M'))$ and $\frac{\lvert s\rvert_v}{\lvert s\rvert'_v}$ for each embedding $v\colon F\to \bC$. 
(The metrics $\lvert\cdot\rvert$ and $\lvert \cdot \rvert'$ are determined by $M$ and $M'$ respectively.) 

It is easy to see that $\lvert s\rvert_v=(\prod_p \Deg_{\et,p})^{\frac{w}{2}}\lvert s\rvert'_v$. 
On the other hand, by using Proposition ~\ref{exact sequence}, one can show that 
\[
\frac{\#L(M)_{\torsion}}{\#L(M')_{\torsion}}\cdot\#(L(M)/L(M'))=\prod_v \prod_r (\#\gr^r(\fM_v/\fM'_v))^r. 
\]
Finally, we remark that the proof of Lemma ~\ref{exponent} can be modified to show
\[\prod_r (\#\gr^r(\fM_v/\fM'_v))^r=\Deg_{\dR,v}, 
\]
and we are done. 
\end{proof}

\subsection{A definition of goodness}
Keep the notation, and fix a prime number $p$. 
Suppose that $f$ is a $p$-isogeny, namely, the $\ell$-adic part of $f$ is an isomorphism for $\ell\neq p$. 
We further assume that the cokernel $T_p/T'_p$ of the $p$-adic part of $f$ is a free $\bZ/p^m\bZ$-module of rank $h$. 

Set $d_{\BK, v}=\frac{\log_p (\Deg_{\dR,v})}{m\cdot [F_v\colon \bQ_p]}$. 

\begin{defn}\label{good isogeney}
In the above situation, 
$f$ is called a good $p$-isogeny if the Galois representation $\det (T_p/T'_p)$ satisfies the following conditions:
\begin{itemize}
\item It is unramified at any finite place $\lambda$ above $\ell\neq p$. 
\item For any finite place $v$ above $p$, $d_{\BK, v}$ is an integer.  
\item For any finite place $v$ above $p$, its restriction to $I_v$ is isomorphic to a power of the cyclotomic character with the exponent $-d_{\BK,v}$. 
\end{itemize}
\end{defn}

\begin{rem}
Our definition of good $p$-isogenies is inspired by Raynaud's notion of \emph{belles $p$-isog\'enies} between abelian varieties~\cite{Raynaud}*{4.1.1}. 
However, even for the first cohomology of abelian varieties, a class of good $p$-isogenies is broader than that of belles $p$-isog\'enies. 
We adopt such a definition because Raynaud's main arguments in ~\cite{Raynaud} essentially still work for this class. 
\end{rem}

Recall the transfer homomorphism $\Ver\colon G^{\ab}_{\bQ}\to G^{\ab}_{F}$. 
We regard $\det (T_p/T'_p)$ as a representation of $G_\bQ$ via $\Ver$.

Define $d_{\BK, p}=\sum_{v|p}[F_v\colon \bQ_p]\cdot d_{\BK,v}$.
Note that $d_{\BK,p}\leq [F\colon \bQ]\cdot hw$. 

 \begin{lem}\label{Ver}
As representations of $G_\bQ$, $\det (T_p/T'_p)\cong \bZ/p^m\bZ(-d_{\BK,p})$. 
\end{lem}

\begin{proof}
By the class field theory, the transfer $\Ver$ has an adelic description and it corresponds to the inclusion map between idele groups. 
Since $f$ is a good $p$-isogeny, one can use such adelic description to show that $\det (T_p/T'_p)$ is unramified at all finite places of $\bQ$ except $p$ and its restriction to $I_p$ is isomorphic to $\bZ/p^m\bZ(-d_{{\BK},p})$, cf. ~\cite{Raynaud}*{4.2.8}. 
Therefore, it is isomorphic to $\bZ/p^m\bZ(-d_{{\BK},p})$ because there is no nontrivial unramified extension of $\bQ$. 
\end{proof}

\begin{rem}
Consider the induced representation $\Ind^{G_\bQ}_{G_F}(T_p/T'_p)$ and its determinant. 
Also, let $\varepsilon$ be the determinant of $\Ind^{G_\bQ}_{G_K}\bZ$. 
This is a character valued in $\{1, -1\}$.  
It is known that $\varepsilon^{-h}\cdot \det (\Ind^{G_\bQ}_{G_F}(T_p/T'_p))$ is isomorphic to $\det (T_p/T'_p)$ as representations of $G_{\bQ}$. 
\end{rem}

\subsection{Preservation of heights}
As a consequence of the purity, or the Weil conjecture ~\cite{Deligne}, we prove that good isogenies preserve heights under certain conditions. 
We include discussion on semistable bad reduction, because it gives conjectural better bounds.  

\begin{defn}
Let $\ell$ be a prime number. 
\begin{enumerate}
\item
Let $K$ be a $\ell$-adic local field and $V$ a $p$-adic representation of $G_K$ with $p\neq \ell$. 
We say that $V$ is pure of weight $w$ if the associated Weil-Deligne representation (see, for instance,~\cite{Taylor-Yoshida}*{p.469}) is pure of weight $w$ in the sense of Taylor-Yoshida ~[\textit{ibid.}, p.471]. 

A $\bZ$-motive $M$ (or $\bQ$-motive $M_\bQ$) over $F$ has pure reduction at $\ell$ if, for any place $\lambda$ of $F$ dividing $\ell$,  
$M_p$ is pure of weight $w_p$ as a representation of $G_{F_\lambda}$ for any $p\neq \ell$. (So, the weight $w_p$ is independent of $\lambda$.)
\end{enumerate}
\end{defn}

\begin{rem}
The weight-monodromy conjecture implies that $M$ has pure reduction at any $\ell$. 
\end{rem}

\begin{lem}\label{pure}
\begin{enumerate}
\item The set of prime numbers where $M$ does not have pure reduction is finite. 
\item Suppose $M$ has pure reduction at $\ell$. Then $w_p=w$. 
\item If $M$ has pure reduction at $\ell$ and good reduction at $\lambda$, then the characteristic polynomial of the geometric Frobenius $\Frob_\lambda$ is independent of $p$ and has coefficients in $\bQ$. Moreover, coefficients are integers if the Hodge-Tate weights are nonpositive. 
\end{enumerate}
\end{lem}

\begin{rem}
Assuming the truth of the weight-monodromy conjecture, it is expected that the conclusion of (3) holds for all finite places. 
In fact, existence of K\"unneth projectors would imply it if $M$ is effective, see ~\cite{Saito}. 
If $M$ is not effective, it seems we need to use expected functoriality of the $p$-adic weight spectral sequence which is not yet published. 
(The $\ell$-adic counterpart is known ~\cite{Saito}.)
\end{rem}

\begin{proof}
From the Weil conjecture ~\cite{Deligne}, (1) follows. 
One can use transfer homomorphism $\Ver$ to prove (2) as in the argument before. 
Independence and integrality are well-known consequences of the Weil conjecture if $M$ is effective. 
(These are not really straightforward consequences because correspondences are involved, cf.~\cite{KM}.)
Even if $M$ is not effective, independence obviously still holds. 
Also, integrality up to $\ell$-powers can be seen easily. 
To deal with $\ell$-powers, we use crystalline cohomology and independence at $\ell=p$. 
Then, it follows from the fact that the Newton polygon lies above the Hodge polygon because the Hodge-Tate weights are nonpositive.  
A similar, but more detailed, argument can be found in the proof of ~\cite{KW}*{Lemma 2.2}. 
\end{proof}

Go back to the previous situation where an isogeny $f\colon T'\to T$ is given. 
Choose a prime number $\ell$ so that
\begin{itemize}
\item $\ell$ is different from $p$, 
\item $M_\bQ$ has good reduction at any places above $\ell$, and 
\item $M_\bQ$ has pure reduction at $\ell$. 
\end{itemize}

Note that such a prime number exists by Lemma ~\ref{pure}. 
For each place $\lambda$ above $\ell$, we consider the transfer $\Ver_\lambda\colon G^{\ab}_{\bQ_{\ell}}\to G^{\ab}_{F_\lambda}$. 
Then we regard $\bigwedge^h M_p$ as a representations of $G_{\bQ_{\ell}}$ via $\Ver_\lambda$ and denote it by $V_{p, \lambda, h}$. 
We further take the tensor product $V_{p, \ell, h}=\bigotimes_\lambda V_{p, \lambda, h}$ and write $P_{\ell, h}(t)$ for the characteristic polynomial $\det (1-\Frob_\ell t; V_{p,\ell, h})$ of the geometric Frobenius $\Frob_\ell$. 
The polynomial $P_{\ell, h} (t)$ has coefficients in $\bZ$ and it is independent of $p$. 
Moreover, for any embedding into $\bC$, the complex absolute value of any root of $P_{{\ell, h}}(t)$ is $\ell^{[F\colon \bQ]\cdot \frac{hw}{2}}$. 

By Lemma ~\ref{Ver}, $\ell^{d_{\BK,p}}$ modulo $p^m$ is an eigenvalue of $\Frob_\ell$ acting on
$\det (T_p/T'_p)$. 

\begin{lem}
$P_{\ell, h}(\ell^{d_{\BK,p}}) \equiv 0 \mod p^m$. 
\end{lem}

\begin{proof}
For each $\lambda$, there is an exact sequence of representations of $G_{F_\lambda}$
\[
0\to T'_p/(p^m T_p\cap T'_p)\to T_p/p^m T_p\to T_p/T'_p\to 0.  
\]
This exact sequence splits as $\bZ/p^m\bZ$-modules because $T_p/T'_p$ is free. 
So, $\det (T_p/T'_p)$ is a quotient representation of $\bigwedge^h (T_p/p^m T_p)$ which is a direct summand as $\bZ/p^m \bZ$-module. 

For each $\lambda$, we regard $\det (T_p/T'_p)$ as a representations of $G_{\bQ_\ell}$ via $\Ver_\lambda$. 
Their tensor product through $\lambda$ is isomorphic to $\bZ/p^m\bZ(-d_{\BK, p})$. 
(Use the relation among $\Ver$ and $\Ver_\lambda$, and apply Lemma ~\ref{Ver}.)

Therefore, $\bZ/p^m \bZ(-d_{\BK, p})$ is a quotient representation of $V_{p,\ell, h}$ which is a direct summand as a $\bZ/p^m\bZ$-module. 
This finishes the proof.  
\end{proof}

Let $I_h$ be the set of all $\ell$-powers $\ell^\mu$ such that $\mu\leq [F\colon \bQ]\cdot hw$ and $\mu\neq [F\colon\bQ]\cdot \frac{hw}{2}$. 
Then, $P_{\ell,h}(\ell^\mu)$ for $\mu\in I_h$ is a nonzero integer. 
Define an integer $m_{p,h}$ to be the least nonnegative integer such that $p^m$ does not divide $P_{\ell, h}(\ell^\mu)$ for any $m> m_{p, h}$ and $\mu \in I_h$.  

\begin{thm}\label{preservation}
If $m> m_{p,h}$, then $h(M)=h(M')$. 
\end{thm}

\begin{proof}
Recall that $d_{\BK,p}\leq [F\colon \bQ]\cdot hw$. 
Then, it follows from the previous lemma that $d_{\BK,p}=[F\colon\bQ]\cdot\frac{hw}{2}$. 

Now we can use Proposition ~\ref{formula}. 
Since we have equalities
\[\Deg_{\dR,v}=p^{[F_v\colon\bQ_p]\cdot m\cdot d_{\BK,v}}\quad\textnormal{and}\quad \Deg_{\et,p}=p^{mh}, 
\] 
we are done. 
\end{proof}

\begin{rem}\label{height difference}
The proof above also shows that an inequality $\lvert h(M)-h(M')\rvert \leq\frac{hmw}{2}\cdot \log p$ is true in general. 
\end{rem}

Note that for a fixed $\ell$, $m_{p,h}=0$ for all but finitely many $p$. 
These exceptional prime numbers are bounded by some constant depending on $\ell, h, n, w$ and $[F\colon\bQ]$. 
Here, $n$ is the dimension of realizations of $M_{\bQ}$. 
Furthermore, we can bound nonzero $m_{p,h}$.

\begin{prop}
If $m_{p,h}\neq 0$, then $m_{p,h}\leq C\cdot \log_p(2\ell^{[F\colon\bQ]\cdot hw})$, where
$C={n\choose h}^{\#\{\lambda|\ell\}}$. 
In particular, fixing $\ell$, $m_{p,h}$ can be bounded uniformly for $p\neq\ell$ and $h$. 
\end{prop}

\begin{proof}
The inequality follows from the fact that the polynomial $P_{\ell, h}$ has degree $C$ and the complex absolute values of its roots are $\ell^{[F\colon\bQ]\cdot\frac{hw}{2}}$. 
Also use that $\mu\leq [F\colon \bQ]\cdot hw$ for $\mu\in I_h$. 
The bound is uniform because $h\leq n$.   
\end{proof}

\begin{rem}\label{purity at bad primes}
We assumed $M_\bQ$ has good reduction at any places above $\ell$. 
This assumption is only used in the proof above. 
Even if it has bad reduction, the purity assumption implies that the complex absolute values are powers of $\ell$ with exponents bounded by some constant depending on $n, w$ and $[F\colon \bQ]$. 
\end{rem}

\subsection{Fontaine-Laffaille case}

\begin{prop}
Assume that the extension $F\supset \bQ$ is unramified at $p$, $nw<p-1$ and $M$ has good reduction at all places above $p$. 
Then, any $p$-isogeny $f\colon T'\to T$ can be factored into good isogenies.  
\end{prop}

\begin{proof}
It suffices to show that  if $T_p/T'_p$ is a free $\bZ/p^m\bZ$-module of rank $h$, $f$ is good, cf. ~Subsection 9.1. 

For any finite place $\lambda$ above $\ell\neq p$, $\det (T_p/T'_p)$ is obviously unramified by the assumption. 
For any $v$ above $p$, the assumption enable us to apply Fontaine-Laffaille theory, 
Proposition ~\ref{Fontaine-Laffaille:determinant}. It follows that $f$ satisfies the conditions of goodness at $v$. 
\end{proof}

\begin{cor}\label{big p}
Fix a prime number $\ell$ as before, and suppose $m_{p}=0$. 
Furhter, assume that the extension $F\supset \bQ$ is unramified at $p$, $nw<p-1$ and $M$ has good reduction at all places above $p$. 
Then, any $p$-isogeny preserves heights. 
\end{cor}

\subsection{Level $1$ case: tame calculation}
Even if $p$ is ramified or $M$ has bad reduction at a place above $p$, there is a chance to prove any $p$-isogenies preserve heights. 
We work with the tame fundamental character of level $1$ in place of the cyclotomic character. 

Suppose $T_p/T'_p$ is a free $\bZ/p\bZ_p$-module of rank $h$. 
We let $d_{\BK, v}$ be $\rk_{W_1}(\fM'_v/\fM_v)/\varphi^*(\fM'_v/\fM_v)$. 

\begin{prop}
As representations of $I_v$, $\det (T_p/T'_p)\cong \theta_{1}^{-d_{\BK,v}}$. 
\end{prop}

\begin{proof} 
The inertia action on $\det(T_p/T'_p)$ factors through the tame inertia group. 
To compute it, we can add $p$-power roots of a uniformizer to the base field. 
Then, use Remark ~\ref{tame:determinant}. 
\end{proof}

Set $d_{\BK,p}=\sum_v d_{BK,v}\cdot f_v$. 
We regard $\det T_p/T'_p$ as a representation of $G_\bQ$ through $\Ver$.  

\begin{prop}
$\det T_p/T'_p\cong \bZ/p\bZ(-d_{\BK,p})$. 
\end{prop}

\begin{proof}
This is similar to Lemma ~\ref{Ver}. 
Note that the representations is automatically unramified at any finite place $\lambda$ above $\ell\neq p$. 
\end{proof}

Take an auxiliary prime number $\ell$ as before. 
Recall the characteristic polynomial $P_{\ell, h}(t)$. 
One can prove the following statements by repeating the arguments before:

\begin{lem}
$P_{\ell,h}(\ell^{d_{\BK,p}})\equiv 0\mod p$. 
\end{lem}

\begin{prop}\label{tame}
If $p>(2\ell^{[F\colon\bQ]\cdot hw})^C$, then $h(M)=h(M')$.  
\end{prop}

\section{Boundedness}
The main theorem of this section is that heights in a single isogeny class is bounded and, in fact, we obtain an effective bound for differences of heights in a isogeny class. 

To prove such statements, it suffices to show that there are enough good isogenies, which we formulate as a factorization theorem. 
It roughly says that any isogeney is a composition of good isogenies and isogenies with small degrees, see Theorem~\ref{factorization} for the precise statement. 
Torsion $p$-adic Hodge theory plays a key role in the proof of the factorization theorem. 

\subsection{First factorization}
Let $M=(M_\bQ, T)$ and $M'=(M_\bQ, T')$ be a pure $\bZ$-motives of weight $w$ over a number field $F$. 
We assume $M_\bQ$ has semistable reudction everywhere and nonpositive Hodge-Tate weights.  
We also fix a prime number $p$ and suppose $f\colon T'\to T$ is a $p$-isogeny. 

As a $\bZ$-module, the cokernel $T_p/T'_p$ has the following form;
\[
(\bZ/p^{m_1}\bZ)^{h_1}\oplus\cdots\oplus (\bZ/p^{m_a}\bZ)^{h_a}, 
m_1<\cdots < m_a, h_1+\cdots +h_a\leq n. 
\]
Here, $n$ is the dimension of realizations of $M_\bQ$. 
Also, we set $m_0=0$. 
Define $T^i_{p}=T'_p+p^{m_i}T_p$ and $f_i\colon T^{i}_p\to T^{i-1}_p$ to be the natural inclusions. 
Then, $f=f_1\circ\cdots\circ f_a$ and the cokernel of $f_i$ is isomorphic to
$(\bZ/p^{(m_i-m_{i-1})}\bZ)^{h_i+\cdots +h_a}$. 

It is clear that this factorization respects $G_F$-actions and makes sense at the level of $\bZ$-motives. 

\subsection{A lemma at $\ell\neq p$}
Assume that $T_p/T'_p$ is a free $\bZ/p^m\bZ$-module of rank $h$. 

\begin{lem}\label{ell neq p}
Let $\lambda$ be a finite place above $\ell\neq p$ and $m'$ the greatest integer less than or equal to $\frac{m}{hn}$.
Then, $\det (T_p/(T'_p+p^{m'}T_p))$ is unramified at $\lambda$. 
(Note that $T_p/(T'_p+p^{m'}T_p)$ is a free $\bZ/p^{m'}\bZ$-module of rank $h$.)
\end{lem}

\begin{proof}
Take $\sigma\in I_{\lambda}$. 
Since the intertia action is unipotent, $(\sigma-1)^n$  acts by zero on $T_p$. 
Therefore, it does on $T_p/T'_p$ as well. 
Then, it is easy to see that $(\sigma-1)^{hn}$ acts by zero on $\det (T_p/T'_p)$.
(This is not the best bound.)
Since it is a rank $1$ free $\bZ/p^m\bZ$-module, it follows that $\sigma-1$ acts by zero modulo $p^{m'}$.  
\end{proof}
 
\begin{rem}
If $M$ is effective and appeared in the $w$-th cohomology $H^{w}(X)$ of a projective smooth variety $X$ over $F$, then
$(\sigma-1)^{w+1}$  acts by zero on $M_p$. (We are still assuming $M$ has semistable reduction at $\lambda$.) 
Indeed, by Poincar\'e duality and weak Lefschetz, we may assume $\dim X=w$. 
Using alterations, it is reduced to the case $X$ has strictly semistable reduction at $\lambda$ after replacing $F$ by a finite extension if necessary. 
Then, it is a well-known consequence of the weight spectral sequence. 
\end{rem}

\subsection{A factorization theorem}
We keep the assumption that $T_p/T'_p$ is a free $\bZ/p^m\bZ$-module of rank $h$. 

For any finite place $v$ above $p$, setting $K=F_v$, we can apply the results in Section $7$. we write $\delta_{v,h}$ for the corresponding $\delta$ appeared in Corollary~\ref{free}. 
Set
\[
\delta_{p,h}=\max \{\delta_{v,h}\mid v~\textnormal{above}~p \}~\textnormal{and}~c''=\max \{hn, 2^{hw-1}\}.
\] 

\begin{thm}\label{factorization}
Any $f$ as above can be factored into $c''$ good isogenies, a good isogeny of degree $\leq p^{c'' h}$ and two isogenies of degree $\leq p^{h\delta_{p,h}}$. 
\end{thm}

\begin{proof}
By Corollary ~\ref{free}, it suffices to show the following: 
if $\fM'_v/\fM_v$ is a free $\fS_m$-modules for any finite place $v$ above $p$, 
$f$ can be factored into $c''$ good isogenies and a good isogeny of degree $\leq p^{c''h}$. 

This can be seen from Theorem ~\ref{determinant} and Lemma ~\ref{ell neq p}. 
The following is also used; let $\chi$ be a character of $I_{v}$ and suppose its restriction to $G_{K_\infty}\cap I_v$ is a power of the cyclotomic character, then $\chi$ itself is a power of the cyclotomic character with the same exponent. 
\end{proof}

\begin{rem}
In ~\cite{Raynaud}*{4.4.8}, Raynaud proved a similar theorem for the Faltings heights of abelian varieties. His bounds depend on the number of places where an abelian variety has bad reduction, but our bound does not depend on places where $M$ has bad reduction. 
\end{rem}

\begin{cor}\label{bound in the free case}
\[
\lvert h(M)-h(M')\rvert \leq [(m_{p,h})c''+\min\{m_{p, h}, c''\}+ 2\delta_{p,h}]\frac{hw}{2}\cdot\log p.
\]
\end{cor}

\begin{proof}
Use Theorem ~\ref{preservation} and Remark ~\ref{height difference}. 
\end{proof}

\subsection{Boundedness for $p$-isogenies}

We still fix the prime number $\ell$.

Define $m_{p}=\max_h m_{p,h}$. 
Combining Corollary ~\ref{bound in the free case} and the arguments in Subsection 9.1, 
we obtain the following:

\begin{thm}\label{p-isogeny}
For any $p$-isogeny $f\colon T'\to T$, 
\begin{align*}
\lvert h(M)-h(M')\rvert &\leq \\
&(m_p+1)n\cdot \max\{n^2, 2^{nw-1}\}\frac{w}{2}\log p \\
&+(\log_p(nw)+\log_p[F\colon\bQ]+2)[F\colon\bQ]\cdot \frac{n^2w^2}{p-1}\log p. 
\end{align*}
\end{thm}

\begin{rem}\label{ell'}
For all but finitely many $p$, $m_p=0$. These exceptional finitely many prime numbers depend on the choice of $\ell$. 
If $m_{p}\neq 0$, then $m_{p}\leq 2^{n[F\colon\bQ]}\cdot \log_p(2\ell^{[F\colon\bQ]\cdot nw})$. 
So, for a fixed $\ell$, we have bounds for $p$-isogenies if $p\neq\ell$ and these bounds depend on $\ell, p, n, w$ and $[F\colon\bQ]$.  
To obtain a bound for $\ell$-isogenies, we choose another suitable prime number $\ell'\neq\ell$, and apply the above argument to $(\ell',\ell)$ instead of $(\ell,p)$. 
\end{rem}

\subsection{Boundedness for all isogenies}

\begin{thm}
Let $M$ be a pure $\bZ$-motive of weight $w$ over $F$. 
We denote by $R$ the set of prime numbers where $F$ is ramified and by $S$ the set of prime numbers where $M$ does not have pure reduction.  
There exist a constant $C$ depending on $[F\colon\bQ], n, w$ and $S$ 
 such that for any isogeny $f\colon T'\to T$, 
\[
\lvert h(M)-h(M')\rvert \leq C. 
\]
In fact, the set of heights $\{h(M')\mid\textnormal{$M'$ is isogenesous to $M$}\}$ is finite. 
\end{thm}

\begin{proof}
The difference $h(M)-h(M')$ is preserved by Tate twists, so it suffices to consider the case $M$ has the nonpositive Hodge-Tate numbers. 

Then, this follows from Theorem ~\ref{p-isogeny}, Corollary ~\ref{big p} and Proposition ~\ref{tame}. 
Also, we refer to Remark ~\ref{ell'}. 
We emphasize that $\ell$ and $\ell'$ depend on $S$. 
\end{proof}

\begin{rem}
As we mentioned before, under certain conjectures, it is reasonable to allow any prime number to be $\ell$. 
Also, given any $\ell$, we can still control the bound, see Remark ~\ref{purity at bad primes}. 
In particular, the constant $C$ would be independent of $S$. 
For the first cohomology of abelian varieties, it is unconditional and this improves Raynaud's bound in ~\cite{Raynaud}*{4.4.9}. 
\end{rem}

\section{A Finiteness Conjecture}

\subsection{Statement of the conjecture}
Take nonnegative integers $h^{r,w-r}$ so that $h^{r,w-r}=0$ for $r<0$ and $r>w$. 

The following finiteness conjecture is due to Kato ~\cite{Kato1}*{4.2}. 

\begin{conj}[Finiteness]
Fix a real number $c$. 
There are finitely many isomorphism classes of pure $\bZ$-motives $M$ over $F$ of weight $w$ satisfying 
\begin{enumerate}
\item $M$ has semistable reduction at any finite places of $F$. 
\item $\dim_F \gr^r(M_{\dR})=h^{r,w-r}$. 
\item $h(M)< c$. 
\end{enumerate}
\end{conj}

\begin{rem}\label{strong finiteness}
For motives which may not have semistable reduction everywhere, the statement still makes sense if we use stable heights, but it is not expected to be true.  
If one uses Kato's height instead of our stable height, there is a chance for it to be true. 
\end{rem}

There is a variant allowing all $F$ with bounded degrees over $\bQ$. 

The conjecture particularly implies that heights of motives satisfying (1) and (2) should be bounded below. 

The conjecture seems extremely hard, but there are a few cases we can verify weaker statements as we discuss below. 
Also, recall that one of our theorem is a version of the conjecture inside certain families. 

\begin{exam}
Heights of Artin $\bZ$-motives with semistable reduction are all zero. 
They correspond to Galois stable lattices in unramified Artin $\bQ$-representations. 
It is well-known that there are finitely many isomorphism classes of unramified Artin $\bQ$-representations of dimension $n$. 
So, by Jordan-Zassenhaus theorem,  there are finitely many isomorphism classes of Galois stable lattices in unramified Artin $\bQ$-representations of dimension $n$. 

It is conjectured that a $\bQ$-motive $M$ with $\dim \gr^r(M_{\dR})=0$ for $r\neq 0$ is an Artin motive.  A variant is known at the level of suitable realizations, see ~\cite{KW} for details. 
\end{exam}

\begin{exam}
For the first cohomology of abelian varieties, the conjecture is true. The analogous statement for the Faltings height is true, and it is the most important observation in Faltings' paper ~\cite{Faltings83}. 
(Since we use the moduli space as we explained in Remark~\ref{Siegel}, non-existence of principal polarization causes a technical issue. But this can be settled by Zarhin's trick.)

We remark that, conjecturally, a $\bZ$-motive of weight $1$ with nonpositive Hodge-Tate weights would be attached to the first cohomology of an abelian variety. 
\end{exam}

\begin{exam}
In this example, we consider normalized newforms of weight $k>2$ with rational Fourier coefficients. So, $h^{k-1,0}=h^{0,k-1}=1$ and other Hodge numbers are all zero. 
One can attach pure $\bQ$-motives over $\bQ$ to such newforms, see ~\cite{Scholl}. 
Conjecturally, as a part of the Langlands program, any pure $\bQ$-motives over $\bQ$ with prescribed Hodge numbers should be attached to such newforms. 

Such an associated motive has semistable reduction everywhere if and only if the level of the corresponding newform is square-free and its nebentypus is trival because of the local-global compatibility ~\cite{Saito:compatibility}. 

If $k$ is odd, nebentypus must be nontrivial and such newfomrs do not exist. 
Therefore, the modularity would imply the finiteness conjecture. 

For even weights, the nebentypus is automatically trivial, but it seems difficult to show finiteness.  
If we only consider those who have complex multiplication,  we can prove a few things as in ~\cite{Schutt}. 
Assuming the extended Riemann hypothesis, it is proved that there are finitely many such newforms. (We can ignore twists because levels are square-free.) In fact, it is unconditional for $k=4$ or $6$. 
So, the finiteness conjecture for such associated motives up to isogeny is true. 
Combined with one of our theorem, there would be finite many such $\bZ$-motives by the finiteness conjecture.  
\end{exam}

\begin{rem}
These argument shows, for fixed $k$, heights of associated motives (with complex multiplication if $k$ is even) are bounded. 
This matches Kato's speculation that the above Hodge numbers would force heights to be bounded. 
According to him, this would reflect the fact that there are few families of such motives, which is a consequence of Griffiths transversality. 
(Roughly speaking, the moduli space of such motives would be discrete if it makes sense.) 
His observation is based on the analogy to the Vojta conjectures. 
\end{rem}

\begin{rem}
Suppose $k$ is odd and consider motives which may not have semistable reduction everywhere. 
In this case, newforms have complex multiplication, and the result of ~\cite{Schutt} apply to these case as well. 
This seems to be related to Remark ~\ref{strong finiteness}. 
\end{rem}

\subsection{Relationships with other conjectures}
Recall the following conjectures: 

\begin{conj}[Tate]
For any pure $\bQ$-motives $M_\bQ$ and $M'_\bQ$ over $F$, 
the $p$-adic \'etale realization induces an isomorphism
\[
\Hom(M_\bQ,M'_\bQ)\otimes_{\bQ}\bQ_p \overset{\cong}{\to} \Hom_{G_F}(M_p,M'_p)
\]
 for any prime number $p$. 
\end{conj}

\begin{rem}
In fact, this implies, for any pure $\bZ$-motives $M$ and $M'$, $\Hom(M,M')\otimes_{\bZ}\bZ_p \overset{\cong}{\to} \Hom_{G_F}(T_p,T'_p)$. 
\end{rem}

\begin{conj}[Semisimplicity]
For any $\bQ$-motive $M_\bQ$ over $F$ and a prime number $p$, $M_p$ is a semisimple representation of $G_F$. 
\end{conj}

\begin{rem}
This conjecture is sometimes included in the Tate conjecture, or called the Grothendieck-Serre conjecture. 
\end{rem}

\begin{conj}[Shafarevich]
Fix a set of nonnegative integers ${h^{r,w-r}}$ as before and a set $S$ of prime numbers. Then, there are finite many isomorphism classes of pure $\bZ$-motives of weight $w$ with $\dim \gr^r(M_{\dR})=h^{r,w-r}$ which have pure good reduction outside $S$ and semistable reduction inside $S$.  
\end{conj}

\begin{rem}
This might not be explicitly written elsewhere. 
The statement without mentioning the purity should hold as well because the purity is expected for any prime numbers as we discussed before. 
\end{rem}

\begin{conj}[Conjecture \emph{E}]
Numerical equivalence is equal to ($p$-adic) homological equivalence. 
\end{conj}

\begin{rem}
This is asked by Tate in his Woods Hole talk. 
Now, this is widely known as Grothendieck's standard conjecture \emph{D}. 
\end{rem}

Finally, following Faltings, we relate the finiteness conjecture to these conjectures. 

\begin{prop}
Assume the finiteness conjecture is true. Then, the following holds:
\begin{enumerate}
\item The Tate conjecture is true. 
\item The conjecture E implies the semisimplicity conjecture. 
\item The conjecture E implies the Shafarevich conjecture. 
\end{enumerate}
\end{prop}

\begin{proof}
We essentially repeat Faltings' (or more classical) argument. 
We give only a summary. Given a $G_F$-stable subrepresentation of $M_p$, one can produce a sequence of $p$-isogeneous $\bZ$-motives. The boundedness result and the finiteness conjecture imply that infinity many of them are isomorphic, and one obtains an endomorphism of $M_\bQ$ whose image is the given subrepresentation.  
Then, this implies the Tate conjecture. 

If numerical equivalence coincides with homological equivalence, the category of motives is semisimple ~\cite{Jan}.  
So, the endomorphism algebra of $M$ is semisimple, 
and the endomorphism above can be taken as an idempotent, hence the semisimplicity of $M_p$ follows.

To show the Shafarevich conjecture modulo isogenies, using the Tate conjecture, we may work with semisimple $p$-adic Galois representations.
It is well-known that there are finitely many isomorphism classes of $n$-dimensional semisimple $\bQ_p$-representations of $G_F$ which is unrafimfied and pure of weight $w$ outside $S$. 
So, there are only finite many such isogeny classes. 
Since heights are bounded in each isogeny class, the finiteness conjecture assures that there are finitely many isomorphism classes inside them. 
\end{proof}


\begin{bibdiv}
\begin{biblist}
\bib{Beilinson}{article}{
   author={Beilinson, A.},
   title={On the crystalline period map},
   journal={Camb. J. Math.},
   volume={1},
   date={2013},
   number={1},
   pages={1--51},
   review={\MR{3272051}},
}
\bib{Bhatt}{article}{
 author={Bhatt, Bhargav}, 
 title={p-adic derived de Rham cohomology}, 
 eprint={http://arxiv.org/abs/1204.6560}
}
\bib{Breuil97}{article}{
   author={Breuil, Christophe},
   title={Repr\'esentations $p$-adiques semi-stables et transversalit\'e de
   Griffiths},
   journal={Math. Ann.},
   volume={307},
   date={1997},
   number={2},
   pages={191--224},
   review={\MR{1428871 (98b:14016)}},
}
\bib{Breuil1999}{article}{
   author={Breuil, Christophe},
   title={Une remarque sur les repr\'esentations locales $p$-adiques et les
   congruences entre formes modulaires de Hilbert},
   journal={Bull. Soc. Math. France},
   volume={127},
   date={1999},
   number={3},
   pages={459--472},
   review={\MR{1724405 (2000h:11054)}},
}
\bib{Breuil}{article}{
   author={Breuil, Christophe},
   title={Integral $p$-adic Hodge theory},
   conference={
      title={Algebraic geometry 2000, Azumino (Hotaka)},
   },
   book={
      series={Adv. Stud. Pure Math.},
      volume={36},
      publisher={Math. Soc. Japan, Tokyo},
   },
   date={2002},
   pages={51--80},
   review={\MR{1971512 (2004e:11135)}},
}
\bib{Caruso-Liu}{article}{
   author={Caruso, Xavier},
   author={Liu, Tong},
   title={Quasi-semi-stable representations},
   journal={Bull. Soc. Math. France},
   volume={137},
   date={2009},
   number={2},
   pages={185--223},
   review={\MR{2543474 (2011c:11086)}},
}
\bib{Cattani-Kaplan-Schmid}{article}{
   author={Cattani, Eduardo},
   author={Kaplan, Aroldo},
   author={Schmid, Wilfried},
   title={Degeneration of Hodge structures},
   journal={Ann. of Math. (2)},
   volume={123},
   date={1986},
   number={3},
   pages={457--535},
   review={\MR{840721 (88a:32029)}},
}
\bib{Colmez-Fontaine}{article}{
   author={Colmez, Pierre},
   author={Fontaine, Jean-Marc},
   title={Construction des repr\'esentations $p$-adiques semi-stables},
   journal={Invent. Math.},
   volume={140},
   date={2000},
   number={1},
   pages={1--43},
   review={\MR{1779803 (2001g:11184)}},
}
\bib{Deligne}{article}{
   author={Deligne, Pierre},
   title={La conjecture de Weil. I},
   journal={Inst. Hautes \'Etudes Sci. Publ. Math.},
   number={43},
   date={1974},
   pages={273--307},
   review={\MR{0340258 (49 \#5013)}},
}
\bib{Faltings83}{article}{
   author={Faltings, Gerd.},
   title={Endlichkeitss\"atze f\"ur abelsche Variet\"aten \"uber
   Zahlk\"orpern},
   journal={Invent. Math.},
   volume={73},
   date={1983},
   number={3},
   pages={349--366},
   review={\MR{718935 (85g:11026a)}},
}
\bib{Faltings-Chai}{book}{
   author={Faltings, Gerd},
   author={Chai, Ching-Li},
   title={Degeneration of abelian varieties},
   series={Ergebnisse der Mathematik und ihrer Grenzgebiete (3) [Results in
   Mathematics and Related Areas (3)]},
   volume={22},
   note={With an appendix by David Mumford},
   publisher={Springer-Verlag, Berlin},
   date={1990},
   pages={xii+316},
   review={\MR{1083353 (92d:14036)}},
}
\bib{Faltings99}{article}{
   author={Faltings, Gerd},
   title={Integral crystalline cohomology over very ramified valuation
   rings},
   journal={J. Amer. Math. Soc.},
   volume={12},
   date={1999},
   number={1},
   pages={117--144},
   review={\MR{1618483 (99e:14022)}},
}
\bib{Faltings02}{article}{
   author={Faltings, Gerd},
   title={Almost \'etale extensions},
   note={Cohomologies $p$-adiques et applications arithm\'etiques, II},
   journal={Ast\'erisque},
   number={279},
   date={2002},
   pages={185--270},
   review={\MR{1922831 (2003m:14031)}},
}
\bib{Fontaine-Laffaille}{article}{
   author={Fontaine, Jean-Marc},
   author={Laffaille, Guy},
   title={Construction de repr\'esentations $p$-adiques},
   journal={Ann. Sci. \'Ecole Norm. Sup. (4)},
   volume={15},
   date={1982},
   number={4},
   pages={547--608 (1983)},
   review={\MR{707328 (85c:14028)}},
}
\bib{Griffiths}{article}{
   author={Griffiths, Phillip A.},
   title={Periods of integrals on algebraic manifolds. III. Some global
   differential-geometric properties of the period mapping},
   journal={Inst. Hautes \'Etudes Sci. Publ. Math.},
   number={38},
   date={1970},
   pages={125--180},
   review={\MR{0282990 (44 \#224)}},
}
\bib{Hyodo-Kato}{article}{
  author={Hyodo, Osamu}, 
  author={Kato, Kazuya}, 
  title={Semi-stable reduction and crystalline cohomology with logarithmic poles}, 
  journal={Ast\'erisque}, 
  number={223}, 
  date={1994}, 
  pages={221--268}, 
  review={\MR{1293974}}, 
}
\bib{Illusie}{article}{
   author={Illusie, Luc},
   title={D\'eformations de groupes de Barsotti-Tate (d'apr\`es A.
   Grothendieck)},
   note={Seminar on arithmetic bundles: the Mordell conjecture (Paris,
   1983/84)},
   journal={Ast\'erisque},
   number={127},
   date={1985},
   pages={151--198},
   review={\MR{801922}},
}
\bib{Jan}{article}{
   author={Jannsen, Uwe},
   title={Motives, numerical equivalence, and semi-simplicity},
   journal={Invent. Math.},
   volume={107},
   date={1992},
   number={3},
   pages={447--452},
   review={\MR{1150598 (93g:14009)}},
}
\bib{Kato2}{article}{
 author={Kato, Kazuya}, 
 title={Height of mixed motives}, 
 date={2013}, 
 note={\url{arxiv.org:1306.5693}}, 
}
\bib{Kato1}{article}{
   author={Kato, Kazuya},
   title={Heights of motives},
   journal={Proc. Japan Acad. Ser. A Math. Sci.},
   volume={90},
   date={2014},
   number={3},
   pages={49--53},
   review={\MR{3178484}},
}
\bib{KM}{article}{
   author={Katz, Nicholas M.},
   author={Messing, William},
   title={Some consequences of the Riemann hypothesis for varieties over
   finite fields},
   journal={Invent. Math.},
   volume={23},
   date={1974},
   pages={73--77},
   review={\MR{0332791 (48 \#11117)}},
}
\bib{Kisin}{article}{
   author={Kisin, Mark},
   title={Crystalline representations and F-crystals},
   conference={
      title={Algebraic geometry and number theory},
   },
   book={
      series={Progr. Math.},
      volume={253},
      publisher={Birkh"auser Boston, Boston, MA},
   },
   date={2006},
   pages={459--496},
   review={\MR{2263197 (2007j:11163)}},
}
\bib{KW}{article}{
   author={Kisin, Mark},
   author={Wortmann, Sigrid},
   title={A note on Artin motives},
   journal={Math. Res. Lett.},
   volume={10},
   date={2003},
   number={2-3},
   pages={375--389},
   review={\MR{1981910 (2004d:14018)}},
}
\bib{Kunnemann}{article}{
   author={K{\"u}nnemann, Klaus},
   title={Projective regular models for abelian varieties, semistable
   reduction, and the height pairing},
   journal={Duke Math. J.},
   volume={95},
   date={1998},
   number={1},
   pages={161--212},
   review={\MR{1646554 (99m:14043)}},
}	
\bib{Liu:BT}{article}{
   author={Liu, Tong},
   title={Potentially good reduction of Barsotti-Tate groups},
   journal={J. Number Theory},
   volume={126},
   date={2007},
   number={2},
   pages={155--184},
   review={\MR{2354925 (2008h:14044)}},
}
\bib{Liu:torsion}{article}{
   author={Liu, Tong},
   title={Torsion $p$-adic Galois representations and a conjecture of
   Fontaine},
   journal={Ann. Sci. \'Ecole Norm. Sup. (4)},
   volume={40},
   date={2007},
   number={4},
   pages={633--674},
   review={\MR{2191528 (2010h:11191)}},
}
\bib{Liu:Breuil}{article}{
   author={Liu, Tong},
   title={On lattices in semi-stable representations: a proof of a
   conjecture of Breuil},
   journal={Compos. Math.},
   volume={144},
   date={2008},
   number={1},
   pages={61--88},
   review={\MR{2388556 (2009c:14087)}},
}
\bib{Liu10}{article}{
   author={Liu, Tong},
   title={A note on lattices in semi-stable representations},
   journal={Math. Ann.},
   volume={346},
   date={2010},
   number={1},
   pages={117--138},
   review={\MR{2558890 (2011d:11272)}},
}
\bib{Liu:monodromy}{article}{
   author={Liu, Tong},
   title={Lattices in filtered $(\phi,N)$-modules},
   journal={J. Inst. Math. Jussieu},
   volume={11},
   date={2012},
   number={3},
   pages={659--693},
   review={\MR{2931320}},
}
\bib{Liu:filtration}{article}{
 author={Liu, Tong}, 
 title={Filtration asociated to torsion semi-stable representations}, 
 date={2015}, 
 note={\url{http://www.math.purdue.edu/~tongliu/pub/filtrationnew.pdf}}, 
}
\bib{Liu:compatibility}{article}{
 author={Liu, Tong}, 
 title={Compatibility of Kisin modules for different uniformizers}, 
 date={2015}, 
 note={\url{http://www.math.purdue.edu/~tongliu/pub/compatible.pdf}},
}
\bib{Nakayama}{article}{
   author={Nakayama, Chikara},
   title={Nearby cycles for log smooth families},
   journal={Compositio Math.},
   volume={112},
   date={1998},
   number={1},
   pages={45--75},
   issn={0010-437X},
   review={\MR{1622751 (99g:14044)}},
   doi={10.1023/A:1000327225021},
}
\bib{Niziol:comparison}{article}{
  author={Nizio{\l}, Wies{\l}awa},
   title={Semistable conjecture via K-theory},
   journal={Duke Math. J.},
   volume={141},
   date={2008},
   number={1},
   pages={151--178},
   review={\MR{2372150 (2008m:14035)}},
}
\bib{Niziol}{article}{
   author={Nizio{\l}, Wies{\l}awa},
   title={On uniqueness of $p$-adic period morphisms},
   journal={Pure Appl. Math. Q.},
   volume={5},
   date={2009},
   number={1},
   pages={163--212},
   review={\MR{2520458 (2011c:14058)}},
}
\bib{Peters}{article}{
   author={Peters, C. A. M.},
   title={A criterion for flatness of Hodge bundles over curves and
   geometric applications},
   journal={Math. Ann.},
   volume={268},
   date={1984},
   number={1},
   pages={1--19},
   review={\MR{744325 (85m:14014)}},
}
\bib{Raynaud}{article}{
   author={Raynaud, Michel},
   title={Hauteurs et isog\'enies},
   language={French},
   note={Seminar on arithmetic bundles: the Mordell conjecture (Paris,
   1983/84)},
   journal={Ast\'erisque},
   number={127},
   date={1985},
   pages={199--234},
   review={\MR{801923}},
}
\bib{Saito:compatibility}{article}{
   author={Saito, Takeshi},
   title={Modular forms and $p$-adic Hodge theory},
   journal={Invent. Math.},
   volume={129},
   date={1997},
   number={3},
   pages={607--620},
   review={\MR{1465337 (98g:11060)}},
}	
\bib{Saito}{article}{
   author={Saito, Takeshi},
   title={Weight spectral sequences and independence of $l$},
   journal={J. Inst. Math. Jussieu},
   volume={2},
   date={2003},
   number={4},
   pages={583--634},
   review={\MR{2006800 (2004i:14022)}},
}
\bib{Saito:log}{article}{
   author={Saito, Takeshi},
   title={Log smooth extension of a family of curves and semi-stable
   reduction},
   journal={J. Algebraic Geom.},
   volume={13},
   date={2004},
   number={2},
   pages={287--321},
   issn={1056-3911},
   review={\MR{2047700 (2005a:14034)}},
   doi={10.1090/S1056-3911-03-00338-2},
}		
\bib{Schmid}{article}{
   author={Schmid, Wilfried},
   title={Variation of Hodge structure: the singularities of the period
   mapping},
   journal={Invent. Math.},
   volume={22},
   date={1973},
   pages={211--319},
   review={\MR{0382272 (52 \#3157)}},
}
\bib{Scholl}{article}{
   author={Scholl, A. J.},
   title={Motives for modular forms},
   journal={Invent. Math.},
   volume={100},
   date={1990},
   number={2},
   pages={419--430},
   review={\MR{1047142 (91e:11054)}},
}
\bib{Schutt}{article}{
   author={Sch{\"u}tt, Matthias},
   title={CM newforms with rational coefficients},
   journal={Ramanujan J.},
   volume={19},
   date={2009},
   number={2},
   pages={187--205},
   review={\MR{2511671 (2010c:11052)}},
}
\bib{Taylor-Yoshida}{article}{
   author={Taylor, Richard},
   author={Yoshida, Teruyoshi},
   title={Compatibility of local and global Langlands correspondences},
   journal={J. Amer. Math. Soc.},
   volume={20},
   date={2007},
   number={2},
   pages={467--493},
   review={\MR{2276777 (2007k:11193)}},
}
\bib{Tsuji}{article}{
   author={Tsuji, Takeshi},
   title={$p$-adic \'etale cohomology and crystalline cohomology in the
   semi-stable reduction case},
   journal={Invent. Math.},
   volume={137},
   date={1999},
   number={2},
   pages={233--411},
   review={\MR{1705837 (2000m:14024)}},
}
\bib{Zucker}{article}{
   author={Zucker, Steven},
   title={Remarks on a theorem of Fujita},
   journal={J. Math. Soc. Japan},
   volume={34},
   date={1982},
   number={1},
   pages={47--54},
   review={\MR{639804 (83j:14009)}},
}
\end{biblist}
\end{bibdiv}
\end{document}